\documentclass{amsart}
\usepackage{comment}
\usepackage{graphicx}%
\usepackage{multirow}%
\usepackage{amssymb,amsmath,latexsym,amsthm,amsfonts,bbm,dsfont}
\usepackage[colorlinks=true,citecolor=red,linkcolor=blue,pdfpagetransition=Blinds]{hyperref}
\usepackage{amsthm}%
\usepackage{hyperref}
\usepackage{anysize}
\marginsize{2cm}{2cm}{2cm}{2cm}
\usepackage[title]{appendix}%
\usepackage{natbib}
\usepackage{xcolor}%
\usepackage{textcomp}%
\usepackage{manyfoot}%
\usepackage{booktabs}%
\usepackage{algorithm}%
\usepackage{algorithmicx}%
\usepackage{algpseudocode}%
\usepackage{listings}%
\usepackage{anyfontsize}
\usepackage{lmodern}
\usepackage{mathtools,stmaryrd}
\theoremstyle{plain}%
\newtheorem{theorem}{Theorem}[section]
\newtheorem{proposition}[theorem]{Proposition}
\newtheorem{lemma}[theorem]{Lemma}
\newtheorem{corollary}[theorem]{Corollary}
\theoremstyle{remark}%
\newtheorem{remark}[theorem]{Remark}%
\theoremstyle{definition}%
\newtheorem{definition}{Definition}%
\newcommand{\E}{\mathbb{E}}
\begin{document}
\title[Controllability for semi-discrete stochastic parabolic operator]{Carleman estimate for semi-discrete stochastic parabolic  operators in arbitrary dimension and applications to controllability}
%%%%%%%%%%%%%
\author[R. Lecaros]{Rodrigo Lecaros}
\address[R. Lecaros]{Departamento de Matem\'atica, Universidad T\'ecnica Federico Santa Mar\'ia,  Santiago, Chile.}
\email{rodrigo.lecaros@usm.cl}

\author[A. A. P\'erez]{Ariel A. P\'erez}
\address[A. A. P\'erez]{(Corresponding Author) Departamento de Matem\'atica, Universidad del B\'io-B\'io, Concepci\'on, Chile.}
\email{aaperez@ubiobio.cl}

\author[M. F. Prado]{Manuel F. Prado}
\address[M. F. Prado]{Departamento de Matem\'atica, Universidad T\'ecnica Federico Santa Mar\'ia,  Santiago, Chile.}
\email{mprado@usm.cl}
%%%%%%%%
\subjclass[2020]{93B05, 93B07, 93C20, 65M06}
\keywords{semi-discrete stochastic parabolic equations, controllability, observability,
global Carleman estimate.}
%%%%%%%%%%%%%%%
\begin{abstract}
 This paper considers a semi-discrete forward stochastic parabolic operator with homogeneous Dirichlet conditions in arbitrary dimensions. We show the lack of null controllability  for a spatial semi-discretization of a null-controllable stochastic parabolic system from any initial datum. However, by proving a new Carleman estimate for its semi-discrete backward stochastic adjoint system, we achieve a relaxed observability inequality, which is applied to derivative $\phi$-null controllability by duality arguments.
\end{abstract}
\maketitle
%%%%%%%%%%%%%%%%%%%%
%-------------
\section{Introduction}\label{sec:Introduction}
%%%%%%%%%%%%%%%%%%%%%%%%%%%%%%%%%%%%%%%%%%%%%%%%%%%%%%%%%%%%%%%%%%%

Let $(\Omega, \mathcal{F},\{\mathcal{F}_{t}\}_{t\geq 0}, \mathds{P})$ be a complete filtered probability space on which a one-dimensional standard Brownian motion $\{B(t)\}_{t\geq0}$ is defined such that $\{\mathcal{F}_{t}\}_{t\geq 0}$ is the natural filtration generated by $B(\cdot)$, augmented by all $\mathds{P}$-null sets in $\mathcal{F}$. We denote by $\mathds{F}$ the progressive $\sigma$-field with respect to $\{\mathcal{F}_{t}\}_{t\geq 0}$. Let $H$ be a Banach space, and let $C([0,T];H)$ be the Banach space of all strongly continuous functionals $H$-defined in $[0,T]$. We denote by $L^2_{\mathcal{F}_t}(\Omega,H)$ the space of all $\mathcal{F}_t$-measurable random variables $\zeta$ such that $\E|\zeta|_{H}^2<\infty$; by $L^\infty_{\mathds{F}}(0,T;H)$ the Banach space consisting of all $H$-valued $\mathds{F}$-adapted essentially bounded processes, and by $L^2_{\mathds{F}}(0,T;H)$ the Banach space consisting of all $H$-valued $\mathds{F}$-adapted continuous processes $X$ such that $\E(|X|^2_{C([0,T];H)})<\infty$, with a canonical norm (similarly, one can define $L_{\mathds{F}}^2(\Omega;C^{m}([0,T];H))$ for any positive integer $m$).

Let $T>0$, $G\subset \mathds{R}^{n}\,(n\in \mathds{N})$ be a bounded domain with a $C^{\infty}$ boundary denoted by $\partial G$, and $G_0$ be a nonempty open subset of $G$. Denote by $\chi_{G_0}$ the characteristic function of $G_0$. This paper is devoted to a study of the null controllability for the spatial semi-discrete version of the following stochastic parabolic equation.
\begin{equation}\label{systemofcontrolcontinuos}
dy-\sum_{i=1}^{n}(\gamma_{i} y_{x_i})_{x_i}dt =\left(\sum_{i=1}^{n}a_{1i}y_{x_i}+a_2y+\chi_{G_0}u\right)dt+(a_3\,y+v)dB(t)
\end{equation}
in $ G\times (0,T)$, with $y=0\,\text{on}\, \partial G\times [0,T]$, $y(0)=y_0\,\text{in}\,G$,

and $\gamma_i>0$ with
\begin{equation*}
    \mbox{reg}(\gamma)\triangleq \text{ess}\sup_{\substack{x\in G \\ i=1,\cdots,n}}\left(\gamma_i+\frac{1}{\gamma_i}+\sum_{i=1}^{n}|\partial_{i}\gamma_i|^2\right)<\infty,
\end{equation*}
where $\gamma \triangleq (\gamma_1,\ldots,\gamma_{n})$, $a_{1m}\in L_{\mathds{F}}^{\infty}(0,T;W^{1,\infty}(G)), m=1,2,\ldots,n$, $
a_{3}\in L_{\mathds{F}}^{\infty}(0,T;L^{\infty}(G))$ and \break $a_2\in L_{\mathds{F}}^{\infty}(0,T;L^{n^\ast}(G))$ with 
\begin{equation}\label{conditionofn*}
    \begin{cases}
        n^{\ast}>2,& \mbox{if }\,n=2, \\
        n^{\ast}\geq n, &\mbox{if }\,n> 2.
    \end{cases}
\end{equation}
In the system \eqref{systemofcontrolcontinuos}, the initial state $y_0\in L^2_{\mathcal{F}_0}(\Omega; L^2(G)) $, $y$ is the state variable, and the control variable consists of the pair $(u,v)\in L_{\mathds{F}}^{2}(0,T;L^{2}(G_0)\times L_{\mathds{F}}^{2}(0,T;L^{2}(G))$. 

The study of the null controllability of stochastic parabolic equations has garnered significant attention in recent years. In particular, the exploration of null-controllability for linear stochastic heat equations is studied in \cite{MR1954986}, where the authors noted the complexity of achieving null-controllability for forward equations, finding a partial result due to the special condition of the final condition of the associated backward equation. In \cite{NullControlability-2009}, null controllability for forward parabolic equations was proven by introducing two controls: one in the drift term and another in the diffusion term, by using a stochastic Carleman estimate for backward parabolic equations. Later, in \cite{Qi.Lu:2011}, the control in the diffusion term was omitted, assuming the coefficients of the parabolic operator were independent of the spatial variable and the control is localized in a subset in time and space. This result was obtained using the Lebeau–Robbiano method \cite{MR1362548}, adapted to the stochastic context. As an additional note, using the Carleman inequalities methodology and the introduction of two controls, this approach has been employed to study various cases, including Neumann boundary conditions in \cite{NullControlabilityNeuman}, as well as dynamic boundary conditions \cite{NullControllabilityDinamic}. It has also been applied to fourth-order parabolic equations \cite{NullControllabilityFourtOrder,ZHANG2025113203}; for semilinear parabolic equations, we refer to \cite{HS:LB:P-2023}, and for degenerate equations where the diffusion term takes the form $\gamma(x) = x^\alpha$ with $\alpha \in (0,2)$ and $x \in (0,1)$, we refer to \cite{NullControllabilitydegenerate1}. 

On the other hand, we introduce the notation for the meshes and operators used in the semi-discrete spatial analysis. Let $n\geq 2$, $N \in \mathds{N}$ and $T > 0$ be fixed parameters. Consider the domain $G\triangleq(0,1)^n$ and define the mesh size as $h \triangleq 1/(N+1)$. The one-dimensional Cartesian grid over the interval $(0,1)$ is given by:
\[
\mathcal{K} \triangleq \{ x_i = ih \mid i =1,...,N\}.
\]
Thus, we consider a regular partition of the domain $G$ as $\mathcal{M} \triangleq G \cap \mathcal{K}^n$. %and denote by $\mathcal{M}_0 \triangleq G_0 \cap \mathcal{M}$ the subset associated with a subdomain $G_0 \subseteq G$.

Now, we define the meshes in the direction $e_{i}$, with $\{e_{i}\}_{i=1}^{n}$ the usual base of $\mathbb{R}^{n}$, as:
\begin{equation}
    \mathcal{M}_{i}^{\ast}\triangleq \tau_{+i}\left( \mathcal{M}\right)\cup\tau_{-i}\left( \mathcal{M}\right) \quad\; \text{and}\,\quad\overline{\mathcal{M}}_{ij}\triangleq(\mathcal{M}_{i}^{\ast})_{j}^{\ast}=\mathcal{M}_{ji}^{\ast\ast},
\end{equation}
where $\tau_{\pm i}(\mathcal{M})\triangleq\left\{ x\pm\frac{h}{2}e_{i}\mid x\in \mathcal{M}\right\}$. In addition, we define the boundary and closure of a set $\mathcal{M}$ as follows
\begin{equation*}
    \begin{split}
       \partial_{i}\mathcal{M}\triangleq \overline{\mathcal{M}}_{ii}\setminus \mathcal{M},\quad \partial \mathcal{M}\triangleq \bigcup_{i=1}^{n} \partial_{i}\mathcal{M} \quad\text{and}\;\overline{\mathcal{M}}\triangleq\, \mathcal{M}\cup \partial \mathcal{M}.
    \end{split}
\end{equation*}
We denote as $C(\mathcal{M})$ the set of functions defined on a regular mesh $\mathcal{M}$ to $\mathbb{R}$. 
We define the average and the difference operators as the operators from $C(\overline{\mathcal{M}})$ to $C(\mathcal{M}_{i}^{\ast})$:
\begin{equation*}
\begin{split}
    A_{i}u(x)&\triangleq \frac{1}{2}\left(\tau_{+i} u(x)+\tau_{-i}u(x)\right)\quad\text{and}\quad D_{i}u(x)\triangleq \frac{1}{h}\left(\tau_{+i}u(x)-\tau_{-i}u(x) \right).
    \end{split}
\end{equation*}
where we are using the notation $\tau_{\pm i}$ in the sense of operator (i.e. $\tau_{\pm i}u(x)=u(x\pm\frac{h}{2}e_{i}$)).

Given these definitions, we can obtain a product rule as follows.
\begin{proposition}[{\cite[Lemma 2.1 and 2.2]{boyer-2010-elliptic}}]\label{pro:product}
For $u,v\in C(\overline{\mathcal{M}})$, we have the following identities in $\mathcal{M}^{\ast}_{i}$. For the difference operator
\begin{equation}\label{eq:difference:product}
    D_{i}(u\,v)=D_{i}u\, A_{i}v+A_{i}u\,D_{i}v,
    \end{equation}
similarly, for the average operator \begin{equation}\label{eq:average:product}
     A_{i}(u\,v)= A_{i}u\,A_{i}v+\frac{h^{2}}{4}D_{i}u\,D_{i}v.
\end{equation}
Finally, in $\mathcal{M}$ we have
\begin{equation}\label{eq:averengeanddifference}
    u=A^2_i u-\frac{h^2}{4}D^2_iu.
\end{equation}
\end{proposition}
%------------------
\begin{remark} Note that in particular, in $\mathcal{M}_{i}^{\ast}$ we have several identities that will be fundamental in the development of this work. First, using $u=v$ in the previous Proposition it follows
\begin{equation}\label{eq:derivadacuadrado}
    D_{i}(|u|^{2})=2D_{i}u\,A_{i}u,
\end{equation}
\begin{equation}\label{eq:promediocuadrado}
    A_{i}(|u|^{2})=\left|A_{i}u \right|^{2}+\frac{h^{2}}{4}\left| D_{i}u\right|^{2},
\end{equation}
besides we emphasize two particular inequalities from \eqref{eq:promediocuadrado}
\begin{equation}\label{eq:promedioinequality}
A_{i}(|u|^{2})\geq \left|A_{i}u \right|^{2},
\end{equation}
\begin{equation}\label{eq:diffaveinequality}
    A_{i}(|u|^2)\geq \frac{h^2}{4}\left|D_{i}u\right|^{2}.
\end{equation}
Moreover, in \cite[Proposition 2.2]{AA:perez:2024} the order of the product rule from Proposition \ref{pro:product} is generalized for the difference and average operators.
\end{remark}
%-----------------
Given a set $\mathcal{W}\subseteq \mathcal{M}$, we define the discrete integral as
\begin{equation*}
    \int_{\mathcal{W}} u\triangleq h^{n}\sum_{x\in \mathcal{W}} u(x),
\end{equation*}
and the following $L^{2}_{h}$ inner product for all $u,v\in C(\mathcal{M})$ as
$\displaystyle
    \langle u,v\rangle_{\mathcal{M}}\triangleq \int_{\mathcal{M}}u\,v,$
with the associated norm
\begin{equation*}
    \left\|u \right\|_{L^{2}_{h}(\mathcal{M})}\triangleq \sqrt{\langle u,u\rangle_{\mathcal{M}}}.
\end{equation*}

Thus, for $u\in C(\mathcal{M})$, we define its $L^{p}_{h}(\mathcal{M})$-norm as 
$$\displaystyle
    \left\| u\right\|_{L_{h}^{p}(\mathcal{M})}\triangleq\left\{\begin{array}{cc}\displaystyle
      \max_{x\in \mathcal{M}}\left\{ |u(x)|\right\},& \textrm{if } p=\infty,\\
      \displaystyle
       \left(\int_{\mathcal{M}} |u(x)|^p\right)^{1/p},&\textrm{if } p<\infty.
    \end{array}\right.
    $$
    
And for $u\in C(\overline{\mathcal{M}})$ we define its $W_{h}^{1,p}(\mathcal{M})$-norm as $$\displaystyle
    \|u\|_{W_h^{1,p}(\mathcal{M})}\triangleq \left\{\begin{array}{cc}\displaystyle
         \|u\|_{L^{\infty}_h(\mathcal{M})}+\max_{i=1,\ldots ,n}\|D_{i}u\|_{L^\infty_h(\mathcal{M}_i^\ast)},&  \textrm{if } p=\infty,\\
         \displaystyle\left(\|u\|^p_{L^p_h(\mathcal{M})}+\sum_{i=1}^n\|D_{i}u\|^p_{L^p_h(\mathcal{M}_i^\ast)}\right)^{1/p},
         & \textrm{if } p<\infty.
    \end{array}\right.$$

%---------------------------
To identify one side of the boundary $\partial_{i}\mathcal{M}$ with respect to the $i$-th direction, we define the exterior normal of the set $\mathcal{M}$ in the direction $e_{i}$ as the map $\nu_{i}\in C(\partial\mathcal{M}_{i})$ given by
\begin{equation*}
    \forall x\in\partial_{i}\mathcal{M}, \nu_{i}(x)=\begin{cases} \ \ 1 &  \mbox{if }\tau_{-i}(x)\in \mathcal{M}_{i}^{\ast} \mbox{ and }\tau_{+i}(x)\notin\mathcal{M}_{i}^{\ast},\\ 
    -1 &  \mbox{if }\tau_{-i}(x)\notin \mathcal{M}_{i}^{\ast} \mbox{ and }\tau_{+i}(x)\in\mathcal{M}_{i}^{\ast},\\ 
    \ \ \,0 &  \mbox{otherwise}.
    \end{cases}
\end{equation*}
 We also define the trace operator $t_{r}^{i}$ for $u\in C(\mathcal{M}^{\ast}_{i})$, for all $x\in \partial_{i}\mathcal{M}$, as follows
\begin{equation*}
     t_{r}^{i}(u)(x)\triangleq \begin{cases}
    u(\tau_{-i}(x))\ & \nu_{i}(x)=1,\\
    u(\tau_{+i}(x))& \nu_{i}(x)=-1,\\
     \ \ 0 & \nu_{i}(x)=0.\end{cases}
\end{equation*}
Moreover, for $u\in C(\partial_{i}\mathcal{M}),
$ we define the integral on the boundary $\partial_i \mathcal{M}$ as \break$\displaystyle
    \int_{\partial_{i}\mathcal{M}} u\triangleq h^{n-1}\sum_{x\in\partial_{i}\mathcal{M}} u(x).$
Then, we have the following integral by parts for the difference and average operators.
%-----------------------------
\begin{proposition}[{\cite[Lemma 2.2]{LOPD}}]\label{prop:integralbyparts}
For any $v\in C(\mathcal{M}_{i}^{\ast})$, $u\in C(\overline{\mathcal{M}}_{ii})$, we have for the difference operator
\begin{equation}\label{eq:int:dif}
    \int_{\mathcal{M}}u\,D_{i}v=-\int_{\mathcal{M}_{i}^{\ast}}v\,D_{i}u+\int_{\partial_{i}\mathcal{M}}u\,t_{r}^{i}(v)\nu_{i},
\end{equation}    
and for the average operator
\begin{equation}\label{eq:int:ave}
    \int_{\mathcal{M}}u\,A_{i}v=\int_{\mathcal{M}_{i}^{\ast}}v\,A_{i}u-\frac{h}{2}\int_{\partial_{i} \mathcal{M}}u\,t_{r}^{i}(v).
\end{equation}
\end{proposition}

%------------------------------------------------------------------------------------------------------------------------------------------------------------------------------------------------------%
Given the setting previously introduced, we end this section by formulating the semi-discrete stochastic controlled system of the \eqref{systemofcontrolcontinuos} in the following way. Let us denote $\displaystyle \tilde{\mathcal{P}}_{h}y\triangleq dy-\sum_{i=1}^{n}D_i(\gamma_iD_iy)dt$, then we consider the following semi-discrete system.
\begin{equation} \label{systemofcontrolDiscrete}
\left\{
\begin{aligned}
    \tilde{\mathcal{P}}_h y &= \left( \sum_{i=1}^n A_i D_i(a_{1i} y) + a_2 y + \chi_{G_0 \cap \mathcal{M}} u \right) dt 
    + (a_3 y + v) dB(t)\text{ in}\, Q, \\
    y &= 0 \text{ on }\, \partial Q, \\
    y(0) &= y_0 \text{ in }\, \mathcal{M},
\end{aligned}
\right.
\end{equation}
for any given initial value $y_0\in L_{\mathcal{F}_{0}}^{2}(\Omega; L_{h}^2(\mathcal{M}))$ and controls $(u,v)\in L^2_{\mathds{F}}(0,T;L_{h}^2(G_0\cap\mathcal{M}))\times L^2_{\mathds{F}}(0,T;L_{h}^2(\mathcal{M}))$, with $Q\triangleq \mathcal{M}\times (0,T)$, $Q_i^{\ast}\triangleq \mathcal{M}_i^{\ast}\times(0,T)$ and $\partial Q\triangleq \partial \mathcal{M}\times (0,T)$; where $\gamma_i$ denote the sample of $\gamma^{i}$ onto both the primal and dual meshes, and $a_{1i}, a_{2}, a_{3}, a_{4}$ stands for the evaluation on the primal mesh, respectively. 

In turn, the study of semi-discrete and fully-discrete counterparts of established controllability results has received much attention, particularly in contexts where continuous results do not hold. For example, in general, the discrete versions of the unique continuation property are false since an explicit counterexample is presented in \cite{zuazuaexample} in the deterministic context. Moreover, in \cite{boyer2013penalised}, the potential non-controllability of semi-discrete systems is identified, even when the continuous problem is controllable, due to the presence of uncontrollable eigenmodes in the discrete operator. Adapting the idea from \cite{boyer2013penalised}, it is possible to show that a similar situation occurs for the stochastic setting. Indeed, considering $n=2$, $\gamma_i=1$, $a_{1i}=0$ for $i=1, 2$ and $a_2=a_3=0$ in \eqref{systemofcontrolDiscrete}, the forward semi-discrete system of the parabolic equation can be written as
\begin{equation}\label{eq:example1}
\begin{cases}
     dy-\sum_{i=1}^{2}D_{i}^{2}ydt=\chi_{G_0 \cap \mathcal{M}}udt+vdB(t),  &\text{in } \mathcal{M}\times (0,T), \\
     y=0,  &\text{on }\partial\mathcal{M}\times (0,T). 
\end{cases}
\end{equation}
Let $(x_{i},x_{j})\in\mathcal{M}$. We define
\begin{equation*}
    \psi(x_{i},x_{j})\triangleq\left\{\begin{array}{cl}
     (-1)^i,    & i=j, \\%\,\text{even},  \\
      0,   & i\not=j,%, %\\
      %-1, & i=j\, \text{odd}.
    \end{array}\right.
\end{equation*}
and extend it by zero on $\partial\mathcal{M}$. Notice that
\begin{equation}\label{eq:eigen}
    \sum_{i=1}^{2}D_{i}^{2}\psi=-\frac{4}{h^2}\psi.
\end{equation}
Thus, $\psi$ is an eigenfunction of $\sum_{i=1}^{2}D_{i}^{2}$. Let $\Tilde{\psi}=e^{-4(T-t)/h^2}\psi$. Applying Itô's formula to the map $\displaystyle t\longmapsto\int_{\mathcal{M}}y(t)\Tilde{\psi}(t)$, we obtain
\begin{equation}\label{eq:deduction}
\begin{split}
    \int_{\mathcal{M}}d\left(y(t)\Tilde{\psi}(t)\right)=&\int_{\mathcal{M}}\Tilde{\psi}\,dy+\int_{\mathcal{M}}\partial_{t}(\Tilde{\psi})y\,dt\\
    =&\int_{\mathcal{M}}\Tilde{\psi}\,dy+\int_{\mathcal{M}}\frac{4}{h^2}\Tilde{\psi}\,y\,dt,
\end{split}
\end{equation}
since $\Tilde{\psi}$ is deterministic. Using \eqref{eq:eigen}, we can rewrite the second integral from the right-hand side above, which yields
\begin{align}
         \int_{\mathcal{M}}d\left(y(t)\Tilde{\psi}(t)\right)=&\int_{\mathcal{M}}\Tilde{\psi}\,dy-\int_{\mathcal{M}}\sum_{i=1}^{2}D_{i}^{2}\left(\Tilde{\psi}\right)\,y\,dt \nonumber.
         \end{align}
Now, by using the integration by parts from Proposition \ref{prop:integralbyparts} we get
         \begin{align*}
          \int_{\mathcal{M}}d\left(y(t)\Tilde{\psi}(t)\right)=& \int_{\mathcal{M}}\Tilde{\psi}\,dy+\sum_{i=1}^2\int_{\mathcal{M}_i^{\ast}}D_i(\Tilde{\psi})D_i(y)\,dt-\sum_{i=1}^2\int_{\partial_i \mathcal{M}}y t_r^i(D_i\Tilde{\psi})\nu_i\;dt \nonumber\\
         =&\int_{\mathcal{M}}\Tilde{\psi}(dy-\sum_{i=1}^{2}D_{i}^{2}ydt)+\sum_{i=1}^2\int_{\partial_i \mathcal{M}}\left[\Tilde{\psi}t_r^i(D_iy)-y t_r^i(D_i\Tilde{\psi})\right]\nu_i\;dt.
\end{align*}
Since $y$ satisfies \eqref{eq:example1} and $\Tilde{\psi} = 0$ on $\partial\mathcal{M}$, it follows that
\begin{equation*}
       \int_{\mathcal{M}}d\left(y(t)\Tilde{\psi}(t)\right)=\int_{\mathcal{M}}\Tilde{\psi}\chi_{\mathcal{M}\cap G_0}u\,dt+\int_{\mathcal{M}}\Tilde{\psi}\, v dB(t).
\end{equation*}
Integrating over $(0,T)$ and taking expectations in the above equation, we have 
\begin{equation*}
    \left.\E\int_{\mathcal{M}}y(t)\Tilde{\psi}(t)\right|_0^T=\E\int_{0}^T\int_{G_0 \cap \mathcal{M}}\Tilde{\psi}u\,dt.
\end{equation*}
Let $G_0\subset G$ be such that $G_0 \cap \mathcal{M}$  does not intersect the support of the eigenfunction $\psi$ (i.e. for all $(x_i, x_j) \in G_0 \cap \mathcal{M}$, we have $i \neq j$). For any control pair $(u,v)$, we obtain
\begin{equation*}
    \E\int_{\mathcal{M}}\psi y(T)=e^{-4T/h^2}\E\int_{\mathcal{M}}\psi y(0),
\end{equation*}
and this quantity (which does not depend on the controls $(u,v)$) is nonzero if the initial data $y(0)$ satisfies $\displaystyle \E\int_{\mathcal{M}}\psi y(0)\not=0$. Observe that this pathology does not occur in the continuous setting since, according to \cite{NullControlability-2009}, the system \eqref{systemofcontrolcontinuos} is null-controllable for any initial data in $L^2_{\mathcal{F}_0}(\Omega,L^2(G))$ and this is one of the reasons why the analysis of semi-discrete problems contains new difficulties.

Due to the aforementioned difficulty, we cannot achieve null controllability at the discrete level. Therefore, we will not try to drive the solution to zero at time $T$. Instead, we require that, at time $T$, the norm of the solution be bounded by a function that depends on the discretization mesh size $h$, allowing uniform convergence towards the null controllability of the continuous problem. To this end, we consider a non-decreasing function $\phi(h) \in (0, \infty)$ for $h \in (0, \infty)$ such that $\lim_{h \to 0} \phi(h) = 0$ and for any $\kappa>0 $ satisfies
\begin{equation}\label{condphi}
    \lim\inf_{h\to 0} \frac{\phi(h)}{e^{-\kappa/h}}>0.
    \end{equation}
Thus, we are interested in the following notion of controllability for the semi-discrete stochastic parabolic system \eqref{systemofcontrolDiscrete}.
\begin{definition}
    Let $T>0$ and $G_0\subset G$. The system \eqref{systemofcontrolDiscrete} is said to be $\phi$-\textit{null controllable} if there exists $\tilde{h}>0$ such that for all $h\leq \tilde{h}$ and for any initial data $y_0\in L^2_{\mathcal{F}_0}(\Omega, L_{h}^2(\mathcal{M}))$, there exist control pairs $(u,v)\in L_{\mathds{F}}^{2}(0,T;L_{h}^{2}(\mathcal{M}\cap G_0))\times L_{\mathds{F}}^{\infty}(0,T;L_{h}^{2}(\mathcal{M}))$ uniformly bounded with respect to $h$ such that the solution $y$ of the semi-discrete system \eqref{systemofcontrolDiscrete} satisfies, at time $T$,
    $$
    \E\int_{\mathcal{M}}|y(T)|^2\leq C \phi(h)\E\int_{\mathcal{M}}|y_0|^2,
    $$
    where $\phi$ satisfies \eqref{condphi} and $C>0$ is uniform with respect to $h$.
\end{definition}
\begin{remark}
    The system \eqref{systemofcontrolDiscrete} can be made \textit{$\phi$-null controllable} using only a control $u$ in the drift term, or perhaps with an additional localized control $v$ acting on the diffusion term. This remains an open and interesting problem in both the discrete and the continuous setting. As discussed in \cite[Remark 1.8]{baroun2023nullcontrollabilitystochasticparabolic}, most existing methods rely on Carleman estimates with weight functions that depend only on time and space. We believe that a possible way to address this problem could involve the use of random weight functions. Some results on the null and approximate controllability of forward stochastic heat equations when the coefficients are space-independent with a single control in the drift term using the spectral method can be found in \cite{Qi.Lu:2011}. However, these results are developed within a continuous framework, and their adaptation to a discrete setting remains an interesting direction for future work.
\end{remark}

\par The above notion of $\phi$-null controllability has been obtained for semi-discrete parabolic operators in the deterministic context via Carleman estimates. For instance, in the spatial semi-discrete setting in arbitrary dimensions for the semilinear heat equation, the $\phi$-null controllability is established in \cite{boyer-2014} and in the one-dimensional case for the linear heat equation in \cite{boyer-2010-1d-elliptic}. Moreover, the case of the discontinuous diffusion coefficient is established in  \cite{Thuy}. In turn, the fully-discrete in 1D is studied in \cite{GC-HS-2021}  and \cite{LMPZ-2023} for Dirichlet and dynamic boundary conditions, respectively. A similar controllability issue is considered for time semi-discrete parabolic systems, see for instance \cite{BHS:2020} and \cite{HS2023}. For higher-order operators, we refer to \cite{CLTP-2022} and \cite{kumar:hal-04828698}.

Based on the previous discussion, this work establishes the $\phi$-\textit{null controllability} of the semi-discrete spatial approximation of system \eqref{systemofcontrolcontinuos}, as formulated in system \eqref{systemofcontrolDiscrete}. We improve upon the results presented in \cite{zhao:2024} in several ways. First, our result holds in arbitrary dimension. Second, system \eqref{systemofcontrolDiscrete} considers more general diffusive and potential coefficients. These new considerations require a non-trivial decomposition of the conjugate operator (see Section \ref{sub:conjugated}). Moreover, we also relax the integrability condition on the coefficient $a_2$, as required in \cite{NullControlability-2009}, thereby improving the controllability result of \cite{zhao:2024}. This latter improvement is based on a discrete Sobolev-type inequality, which is presented below. 

\begin{theorem}\label{teo:discreteSobolev}(\textbf{Discrete Sobolev-Type Inequality.}) Let $n>1$. There exists a constant $C$ independent of $h$ such that the following inequality holds:
$$
\|u\|_{L_h^{p^{\ast}}(\mathcal{M)}}\leq C\|u\|_{W_h^{1,p}(\mathcal{M})}\quad \forall u\in W_h^{1,p}(\mathcal{M}),
$$
with $u=0$ on $\partial\mathcal{M}$, where 
\begin{itemize}
\item If $1\leq p<n$, then $p^{\ast}$ is given by $\displaystyle \frac{1}{p^{\ast}}=\frac{1}{p}-\frac{1}{n}$,
\item If $p=n$ then $p^{\ast}\in [p,\infty)$. 
\end{itemize}
\end{theorem}

Consequently, the methodology proposed here can be adapted to study the $\phi$-null controllability of various semi-discrete stochastic parabolic operators by developing semi-discrete Carleman estimates in arbitrary dimensions and theoretical results for the discrete setting.
%%%%%%%%%%%%%%%%%%%%%%%%%%%%%%%%%%%%%%%%%%%%%%%%%%%%%%%%%%%%%%%%%%%%%

The paper is structured as follows. Section \ref{sec:proof:carleman} is devoted to proving the Carleman estimate presented in Theorem \ref{theo:Carleman}. The relaxed observability inequality and the controllability result (Theorem \ref{Teo:Observability inequality} and \ref{theo:nullcontrolability}, respectively) are proved in Section \ref{sec:ObservabilityandControllability}. For the sake of clarity, the results related to the weight functions, a large number of proofs of intermediate estimates, and lemmas are provided in Appendices \ref{sec:weight:function}, \ref{tecgnicalstepcrossproduct}, and \ref{sec:intermediateResults}, respectively. 
%%%%%%%%%%%%%%%%%%%%%%%%%%%%%%%%%%%%%%%%%%%%%%%%%%%%%%%%%%%%%%%%%%%%%%%%%%%%%%%%%%%%%%%%%%%%%%%%%%%%%%%%%%%%%%%%%%%%%%%%%%%%%%%%%%%%%%%%%%%%%%%%%%%%%%%%%%%%%%%%%%%%%%%%
\section{The semi-discrete Carleman estimate}\label{sec:proof:carleman}
To establish the $\phi$-observability inequality for \eqref{systemadjstohocastic}, we will first prove a Carleman inequality associated with the solution \eqref{systemadjstohocastic}. Hence, we introduce the following assumption on our weight function.

%%%%%%%%%%%
\textbf{Assumption:} Let $G_1$ be an arbitrary fixed sub-domain of $G$ such that $\overline{G_1}\subset G_0$. Let $\widehat{G}$ be a smooth open and connected neighborhood of $\overline{G}$ in $\mathbb{R}^n$. The function $x \mapsto \psi(x)$ is in $\mathcal{C}^p(\widehat{G}, \mathbb{R})$, $p$ sufficiently large, and satisfies, for some $c > 0$,
\begin{equation}\label{assumtion:psi}
\psi > 0 \quad \text{in } \widehat{G}, \quad |\nabla \psi| \geq c \quad \text{in } \widehat{G} \setminus G_1, \quad \text{and} \quad \partial_{\nu_i} \psi(x) \leq -c < 0, \quad \text{for } x \in V_{\partial_i G},
\end{equation}
where $V_{\partial_i G}$ is a sufficiently small neighborhood of $\partial_i G$ in $\widehat{G}$, in which the outward unit normal $\nu_i$ to $G$ extends from $\partial_i G$.

For $\lambda\geq 1$ and $K>\|\psi\|_{\infty}$, we introduce the function
\begin{align}\label{funcion-peso-2}
    \varphi(x)&=e^{\lambda\psi(x)}-e^{\lambda K},
\end{align}
 and for $0<\delta < 1/2$,
\begin{equation}\label{theta-delta}
    \theta(t)=\frac{1}{(t+\delta T)(T+\delta T-t)},\quad t\in [0,T].
\end{equation}
Given $\tau\geq 1$ we set
\begin{equation}\label{eq1}
s(t)=\tau\theta(t).
\end{equation}

\begin{remark}
The parameter $\delta$ is chosen such that $0<\delta\leq\frac{1}{2}$ to avoid the singularities at time $t=0$ and $t=T$. Notice that
\begin{equation}\label{eq:theta} 
\underset{t\in [0,T]}{\max} \theta(t)=\theta(0)=\theta(T)=\frac{1}{T^{2}\delta(1+\delta)}\leq \frac{1}{T^{2}\delta}, 
\end{equation}
and $\underset{t\in[0,T]}{\min}\theta(t)= \theta(T/2)=\frac{4}{T^{2}(1+2\delta)^2}$. Also
\begin{equation}\label{eq:theta'}\frac{d\theta}{dt}=2\left(t-\frac{T}{2}\right)\theta^2.
\end{equation}
\end{remark}

We can state our first main result, a semi-discrete Carleman inequality for the discrete backward system \eqref{systemadjstohocastic}.

\begin{theorem}\label{theo:Carleman}
Let $\gamma_{0}>0$, $\psi$ satisfying assumption \eqref{assumtion:psi} and $\varphi$ according to \eqref{funcion-peso-2}. For $\lambda\geq 1$ sufficiently large, there exist $C$, $\tau_{0}\geq 1$, $h_{0}>0$, $\varepsilon_0 >0$, depending on $G_{0}$, $G_{1}$, $\gamma_{0}$, $T$, and $\lambda$, such that for any $\gamma$, with $\mbox{reg}(\gamma)\leq \gamma_{0}$ we have
    \begin{equation}\label{eq:inequalityCarleman} 
 \begin{split}  
J(w)+&\E\int_{Q}s^{3}\varphi^{3}e^{2s\theta\varphi}|w|^{2}dt\\
\leq &C\left(\E\int_{0}^T\int_{ G_0 \cap \mathcal{M}}s^{3}\varphi^{3}e^{2s\theta\varphi}\,|w|^{2}dt+\E\int_{Q}e^{2s\theta\varphi}|f|^2dt+\E\int_{Q}s^2 e^{2s\theta\varphi}|g|^2dt\right.\\
&\left.+h^{-2}\left.\E\int_{\mathcal{M}}e^{2s\theta\varphi}|w|^2\right|_{t=0}+h^{-2}\left.\E\int_{\mathcal{M}}e^{2s\theta\varphi}|w|^2\right|_{t=T}\right),
\end{split}
\end{equation}
for all $\tau\geq \tau_0(T+T^2)$, $0<h\leq h_0$, $0<\delta<1/2$, $sh\leq \epsilon_0$, where $\displaystyle J(w)\triangleq \sum_{i=1}^{n}\E\int_{Q_{i}^{\ast}}s\varphi e^{2s\theta\varphi}\,|D_{i}w|^{2}\,dt+\E\int_{Q}s\varphi e^{2s\theta\varphi}|A_iD_iw|^2\,dt$ and $w$ satisfy $dw+\sum_{i\in\llbracket 1, n \rrbracket} D_i(\gamma_i D_{i}w)\,dt=fdt+gdB(t)$ with $w=0$ on $\partial \mathcal{M}$.
\end{theorem} 
\begin{remark}
Deriving Carleman estimates for continuous stochastic partial differential operators typically relies on a pointwise identity with an appropriate weight (see, e.g., \citep[Theorem 3.1]{NullControlability-2009}).
However, in the discrete case, we construct integral-form identities that strongly exploit the properties of the primal and dual meshes (i.e., $\mathcal{M}$ and $\mathcal{M}^{\ast}$, respectively) and their relationships with discrete operators, particularly as shown in Propositions \eqref{eq:int:ave} and \eqref{eq:int:dif}.
Therefore, obtaining a pointwise identity in the discrete setting presents new challenges. These identities tend to omit information about the mesh nodes, which may lead to potential issues in subsequent analyses.
This problem is especially interesting in the stochastic semidiscrete setting, as it could open the door to addressing new questions and to jointly studying the backward and forward problems, as done in \cite{NullControlability-2009}, among other references.
\end{remark}
\begin{proof} We follow a classical scheme (see \cite{fursikov-1996}) based on conjugating the original operator with a well-chosen exponential function. For the sake of presentation, we split the proof into three steps: First, we write the conjugate operator into two parts, and an additional term $R_{h}$. (see Section \ref{sub:conjugated}). Then we estimate the cross-inner product between these operators (see Section \ref{sub:croos-product}), and as a final stage, we return to the original variable (see Section \ref{sub:backVarible}).

%-----------
\subsection{Conjugated operator}\label{sub:conjugated}
%------------
For simplicity of notation we write $r\triangleq e^{s\varphi}$ and $\rho\triangleq r^{-1}$, where our weight function is defined as in \eqref{funcion-peso-2} and $s$ is given by \eqref{eq1}. The proof of some technical results can be found in Appendix \ref{tecgnicalstepcrossproduct}.  \\
For all $i=1,...,n$, let us consider the functions $\gamma_{i}$ such that $\mbox{reg}(\gamma)<\gamma_{0}$ and the following notation
\begin{equation*}
   \nabla_{\gamma}f\triangleq(\sqrt{\gamma_1}D_1 f_1,\cdots,\sqrt{\gamma_n}D_nf_n)\quad\text{and}\,\quad \Delta_{\gamma}f\triangleq\sum_{i=1}^{n}\gamma_i \partial_{x_i}^{2}f. 
\end{equation*}
\par Let $\displaystyle \mathcal{P}(w)\triangleq\,dw+\sum_{i=1}^{n}D_{i}\left(\gamma_iD_{i}(w)\right)dt=fdt+gdB(t)$. By regarding the weight function $\rho$, the first step is to consider the change of variable $w=\rho z$, and then our task is to split the conjugate operator 
$\displaystyle r\mathcal{P}(\rho z)=rd(\rho z)+r\sum_{i=1}^{n}D_{i}\left(\gamma_{i}D_{i}(\rho z)\right)dt$
into simple terms that we will estimate separately. Indeed, using Ito's formula  (cf. \citep[Theorem 4.18]{klebaner2012introduction}) and noticing that $r\rho=1$, we obtain $ rd(\rho z)=dz+r\partial_t(\rho) z\,dt.$
In turn, by \eqref{eq:difference:product} from Proposition \ref{pro:product}, we have\\
$
    D_{i}\left( \gamma_{i}D_{i}(\rho z)\right)
    =D_{i}(\gamma_{i}D_{i}\rho)A_{i}^{2}z+A_{i}(\gamma_{i}D_{i}\rho)D_{i}A_{i}z+D_{i}A_{i}\rho A_{i}(\gamma_{i}D_{i}z)+A_{i}^{2}\rho D_{i}(\gamma_{i}D_{i}z).
$
Then, applying Proposition \ref{pro:product} and using that $A_{i}\gamma_{i}=\gamma_{i}+h\mathcal{O}(1)$, we have
\begin{align*}
    A_{i}(\gamma_{i}D_{i}z)&=A_{i}\gamma_{i}A_{i}D_{i}z+\frac{h^2}{4}D_{i}\gamma_{i}D_{i}^{2}z=\gamma_i\, A_iD_iz+\frac{h^2}{4}D_i\gamma_i\,D^2_iz+h\mathcal{O}(1)\,A_iD_iz.
 \end{align*} Similarly, we get
 $A_{i}(\gamma_{i}D_{i}\rho)=\gamma_i\, A_iD_i\rho+\frac{h^2}{4}D_i\gamma_i\,D_i^2\rho+h\mathcal{O}(1)\,A_iD_i\rho$ and $D_{i}(\gamma_{i}D_{i}\rho)= 
D_i\gamma_i\,A_iD_i\rho + \gamma_i\, D_i^2\rho+h\mathcal{O}(1)\,D_i^2\rho$.
Thus, the conjugate operator can be written as
 \begin{equation}\label{eq:conjugate}
r\mathcal{P}(\rho z)-R_{h}(z)\,dt=B_1(z)+B_2(z)\,dt+C_1(z)\,dt+C_2(z)\,dt+C_3(z)\,dt,
\end{equation}
where $\displaystyle C_{1}(z)\triangleq \sum_{i=1}^{n}rA_{i}^{2}\rho\,D_{i}(\gamma_{i}D_{i}z)$, $\displaystyle C_{2}(z)\triangleq \sum_{i=1}^{n}\gamma_{i}rD_{i}^{2}\rho A_{i}^{2}z$, $C_{3}(z)\triangleq r\partial_t(\rho) z$, \break $\displaystyle B_{1}(z)\triangleq dz$, $B_2(z)\triangleq\,2\sum_{i=1}^{n}\gamma_{i}rD_{i}A_{i}\rho D_{i}A_{i}z$
and
\begin{equation*}
    \begin{split}
        R_{h}(z)\triangleq\,&\sum_{i=1}^{n}\left(h\mathcal{O}(1)rD_{i}^{2}\rho +D_{i}\gamma_{i}rA_{i}D_{i}\rho\right)A_{i}^{2}z+\sum_{i=1}^{n}h\mathcal{O}(1)rD_{i}A_{i}\rho D_{i}A_{i}z\\
        &+\sum_{i=1}^{n}\frac{h^{2}}{4}D_{i}\gamma_{i}rD_{i}^{2}\rho D_{i}A_{i}z+\sum_{i=1}^{n}\frac{h^{2}}{4}D_{i}\gamma_i \,rD_{i}A_{i}\rho D_{i}^{2}z.
    \end{split} 
\end{equation*}
 Therefore, adding the terms $\displaystyle C_{4}(z)\triangleq\,\frac{h^2}{4}\sum_{i=1}^{n}D_i(\gamma_i D_i(rD_i^2\rho)A_iz)$, $\displaystyle C_5(z)\triangleq\,\frac{h^2}{4}\sum_{i=1}^{n}D_i(D_i(\gamma_irD^2_i\rho)A_iz)$, and $B_3(z)\triangleq\,-2s(\Delta_{\gamma}\varphi)\, z$
to \eqref{eq:conjugate} we have the following identity in $Q$
\begin{equation}\label{eq:opconjugate}
    r\mathcal{P}(\rho z)+M_h (z)dt=C(z)\,dt+B(z),
\end{equation}
where $C(z)\triangleq C_{1}(z)+C_{2}(z)+C_3(z)+C_4(z)+C_5(z)$, $B(z)\triangleq B_1(z)+(B_{2}(z)+B_3(z))
\,dt$ and \break $M_{h}(z)\triangleq C_4(z)+C_5(z)+B_3(z)-R_hz$.

Thus, our next task is to estimate the cross-product. \begin{equation}\label{conmu}
    2\E \int_Q C(z)B(z) =2\E \sum_{i=1}^{5}\sum_{j=1}^{3}\int_Q C_{i}(z)B_{j}(z)\triangleq\,\sum_{i=1}^{5}\sum_{j=1}^{3}I_{ij}. 
\end{equation}

and to bound the left-hand side of \eqref{eq:opconjugate} with respect to the term $M_h z$. In fact,
we notice that from \eqref{eq:opconjugate}, we have \eqref{eq:opconjugate}.
\begin{equation}\label{eq:estimationP(q)}
    2r\mathcal{P}(\rho z)\,C(z)+2M_h(z)\,C(z)\,dt=2|C(z)|^2dt+2C(z)\,B(z).
\end{equation}
Hence, combining the definition of $r\mathcal{P}(\rho z)$ with \eqref{eq:estimationP(q)}, it follows that
\begin{equation*}
    \E \int_Q 2rf\,C(z)\,dt+2\E\int_{Q}
    M_h(z)\,C(z)\,dt=2\E \int_Q |C(z)|^2dt+2\E \int_{Q}C(z)B(z).
\end{equation*}
By using the Young inequality, we notice that
\begin{equation}\label{eq:Estirf}
    \E \int_0^T |rf|^2\,dt+\E \int_{Q}|M_{h}(z)|^2\,dt \geq 2\E\int_{Q} C(z)\,B(z).
\end{equation}

%%%%%%%%%%%%%%%%%%%%%%%%%%%%%%%%%%%%%%%%%%%%%%%%%%%%%%%%%%%%%%%%%%%%%%%%%%%%%%%%%%%%%%%%%%%%%%%%%%%%%%%%%%%%%%%%%%%%%%%%%%%%%%%%%%%%%%%%%%%%%%%%%%%%%%%%%%%%

The next step is to provide an estimate for the right-hand side of \eqref{eq:Estirf}.
%Applying Young's inequality and the triangular inequality, we can bound $r\mathcal{P}(q)$ in $Q$ as follows.
\begin{remark}
    The initial breakdown of the stochastic differential parabolic operator $\gamma$ plays an important role, particularly in the terms $A_i(\gamma_i D_i z)$ and $D_i(\gamma_i D_i z)$. This leads to additional terms that are included in the component $R_h (z)$ due to the product rule applied to the difference and averaging operators. Furthermore, for the same reason, two extra terms referred to as "correction terms" are added during the calculation of the product $C(z)B(z)$, which increase the number of terms in the decomposition and result in more product computations. On the other hand, including the diffusion coefficient $\gamma$ makes it possible to study operators with discontinuous or degenerate diffusion in a discrete space setting, which could be explored further in future research.
\end{remark}

\subsection{An estimate for the cross-product}\label{sub:croos-product}
%%%%%%%%%%%%%%%%%%%%%%%%%%%%
For the product \eqref{conmu}, we will distinguish the terms into three groups: those involving differential $dz$, those involving additional terms, and those involving the differential $dt$. We will obtain a result for each case, which will be presented in the following. For the reader's convenience, the respective proofs are provided in the appendix section \ref{tecgnicalstepcrossproduct}.

We begin by estimating the terms \( I_{11}, I_{21}, \) and \( I_{31} \), which incorporate the differential term \( dz \). This estimation will utilize tools from stochastic calculus, specifically Itô's formulas as presented in Theorems 4.17 and 4.18 of \cite{klebaner2012introduction}, employing the most suitable version for our purposes. The estimation procedure follows the framework outlined in \cite{zhao:2024}. However, we focus on the $n$-dimensional case with $n >1$, which adds complexity because the function $\gamma$ depends on space. This spatial dependence affects estimates involving differential and average operators, such as those used to derive the estimate \eqref{funtiongammaimport} and others. In particular, this function generates two additional terms called 'corrector terms': $C_4 z$ and $C_5 z$, which appear in $I_{11}$ and $I_{21}$, respectively. In this case, we obtain the following lemma combining the estimates of section \ref{subsec:dz}.
\begin{lemma}\label{lem:dz}.
(\textit{Terms involving the differential $dz$.}) For $\tau h (\max_{[0,T]}{\theta})\leq 1$, we have
\begin{equation*}
    \sum_{i=1}^{5}I_{i1}\geq\,\sum_{i=1}^{n} \E\int_{Q_i^{\ast}}\gamma_i|D_i(dz)|^2-X_{1}-Y_{1}
\end{equation*}
where
\begin{equation*}
\begin{split}
    &X_{1}\triangleq\,\sum_{i=1}^{n}\left(\E\int_{Q_{i}^{\ast}}\mathcal{O}_{\lambda}((sh)^2)\,|D_i(dz)|^2+\E\int_{Q_{i}^{\ast}}T\theta(t)\mathcal{O}_{\lambda}((sh)^2)\,|D_iz|^2\,dt\right)\\
    &+\E\int_{Q}(\mathcal{O}_{\lambda}(s^2+(sh)^2)+T\theta^2\mathcal{O}(1))|\,|dz|^2+\E\int_{Q} T(\theta\mathcal{O}_{\lambda}(s^2+(sh)^2)+\theta^3\mathcal{O}(1))\,|z|^2dt
\end{split}
\end{equation*}
and
\begin{equation*}
\begin{split}
    Y_{1}\triangleq\,&\sum_{i=1}^{n}\E\left. \int_{\mathcal{M}_i^*}\mathcal{O}_{\lambda}(1+(sh)^2)\,|D_iz|^2\right|_0^T+\E\left.\int_{\mathcal{M}}(\mathcal{O}_{\lambda}(s^2+(sh)^2)+T\theta^2\mathcal{O}(1))\,|z|^2\right|_0^T
\end{split}
\end{equation*}
\end{lemma}
%%%%%%%%%%%%%%%%%%%%%%%%%%%%%%%%%%%%%%%%%%%%%%%%%%%%%%%%%%%%%%%%%%%%%%%%%%%%%%%%%%%%%%%%%%%%%%%%%%%%%%%%%%%%%%%%%%%%%%%%%%%%%%%%%%%%%%%%%%%%%%%%%%%%%%%%%%%%%%%%%%%%%%%%%%%%%%%%%%%%%%%%
Our next step is to obtain an estimate for the terms involving the so-called ''correction terms'' from \eqref{conmu}. In particular, the terms $C_4z$ and $C_5z$ are inspired by work in \cite{WZ:2024}, where they appear in the product with the term $B_1z$. Additionally, the term $B_3z$ naturally arises when calculating the Carleman inequality in the deterministic case, as demonstrated in \cite{boyer-2010-elliptic}. Therefore, the estimation of the terms $I_{13}$ and $I_{23}$ will be omitted because it does not differ from the deterministic case \cite[see Lemma 3.3 and Lemma 3.6]{boyer-2014}. The estimates presented in Lemma \ref{lem:AdditionalsTerms} are obtained by combining the estimates \eqref{eq:estimatefinalI_42}-\eqref{eq:estimatefinalI53}, $I_{13}$ and $I_{23}$.

\begin{lemma}\label{lem:AdditionalsTerms}
    (\textit{product of the additional terms.}) For $\tau h (\max_{[0,T]}{\theta})\leq 1$, we obtain
    \begin{align*}
        \sum_{i=4}^{5}I_{i2}+\sum_{i=1}^{5}I_{i3}\geq\, \sum_{i=1}^{n}\E\int_{Q_{i}^{\ast}}4s\lambda^2\varphi\gamma_i&|\nabla_{\gamma}\psi|^2\,|D_iz|^2\,dt\\
        &-\E\int_{Q}4s^3\lambda^4\varphi^3|\nabla_{\gamma}\psi|^4\,|z|^2\,dt-X_2
    \end{align*}
where
\begin{equation*}
    \begin{split}
        &X_2\triangleq\, \sum_{i=1}^{n} \E\int_{Q_i^{\ast}}s|\mathcal{O}_{\lambda}((sh)^2)|\,|D_iz|^2\,dt+\sum_{i=1}^{n} \E\int_{Q_i^{\ast}}s\mathcal{O}_{\lambda}(sh)\,|D_iz|^2\,dt \\
        &+\sum_{i=1}^{n}\E\int_{Q_i^{\ast}}(s\lambda\mathcal{O}(1)+\mathcal{O}_{\lambda}(sh)\,|D_iz|^2\,dt+\E\int_{Q} s|\mathcal{O}_{\lambda}((sh)^2)|\,|z|^2\,dt\\
        &+\E\int_{Q}s^2\mathcal{O}_{\lambda}(1)\,|z|^2\,dt+\E\int_{Q}(s^3\lambda^3\varphi^3\mathcal{O}(1)+s^2\mathcal{O}_{\lambda}(1)+s^3\mathcal{O}_{\lambda}(sh))\,|z|^2\,dt\\
        &+\E\int_{Q} (Cs\lambda^2T\theta^2\varphi|\nabla_{\gamma}\psi|^2+s\lambda T\theta^2\varphi\mathcal{O}(1))\,|z|^2\,dt-\E\int_{Q}s\mathcal{O}_{\lambda}((sh)^2)\,|z|^2\,dt.
    \end{split}
\end{equation*}
\end{lemma}
%%%%%%%%%%%%%%%%%%%%%%%%%%%%%%%%%%%%%%%%%%%%%%%%%%%%%%%%%%%%%%%%%%%%%%%%%%%%%%%%%%%%%%%%%%%%%%%%%%%%%%%%%%%%%%%%%%%%%%%%%%%%%%%%%%%%%%%%%%%%%%%%%%%%%%%%%%%%%%%%%%%%%%%%%%%%
Finally, the terms \(I_{12}, I_{22},\) and \(I_{32}\) are similar to those of the deterministic case as discussed in \cite{boyer-2014}, since the temporal variable does not play an important role. In view of this, we do not provide a detailed proof of the estimation of the following lemma. This is not necessary because it only differs from our notation from the one used in \cite{boyer-2014}.
\begin{lemma}\label{lem:dt}
(\textit{Terms involving the differential $dt$.}) For $\tau h (\max_{[0,T]}{\theta})\leq 1$, we obtain
\begin{equation*} 
    \sum_{i=1}^{3}I_{i2}\geq\,\E\int_{Q} 6s^3\lambda^4\varphi^3|\nabla_{\gamma}\psi|^4\,|z|^2\,dt-\sum_{i=1}^{n}\E\int_{Q_{i}^{\ast}}2s\lambda^2\varphi \gamma_{i}|\nabla_{\gamma}\psi|^2\,|D_iz|^2\,dt-X_3-Y_2
\end{equation*}
where
\begin{equation*}
    \begin{split}
        &X_{3}\triangleq\,\sum_{i=1}^{n}\E\int_{Q}|s\lambda\varphi\mathcal{O}(1)+\mathcal{O}_{\lambda}(sh)+s\mathcal{O}_{\lambda}(sh)+s\mathcal{O}_{\lambda}((sh)^{2})|\,|D_{i}A_{i}z|^{2}\,dt\\
    &+\sum_{\substack{i,j=1\\ i\ne j}}^{n}\E\int_{Q_{ij}^{*}}\left|h\lambda\mathcal{O}(sh)+h\mathcal{O}_{\lambda}((sh)^{2})\right|\,|D_{ij}^{2}z|^{2}\,dt\\
    &+\sum_{i=1}^{n}\E\int_{Q^{\ast}_{i}}|s\lambda\varphi\mathcal{O}(1)+s\mathcal{O}_{\lambda}(sh)+h\mathcal{O}_{\lambda}(sh)+s\mathcal{O}_{\lambda}((sh)^{2})+\mathcal{O}_{\lambda}((sh)^2)|\, |D_{i}z|^{2}\\
    &+\E\int_{Q}(s^2\lambda^3\varphi^2\mathcal{O}(1)+s^2\mathcal{O}_{\lambda}(1)+s^2T\theta\mathcal{O}_{\lambda}(1)+s^3\mathcal{O}_{\lambda}((sh)^2))|z|^2\\
    &+\sum_{i=1}^{n}\E\int_{Q}h^{2}|\mathcal{O}_{\lambda}(sh)|\,|D_{i}^{2}z|^{2}\,dt+\sum_{i=1}^{n}\E\int_{Q}\left|h\lambda\mathcal{O}(sh)+h\mathcal{O}_{\lambda}((sh)^{2})\right|\,|D_{i}^{2}z|^{2}\,dt
    \end{split}
\end{equation*}
and
\begin{equation*}
    \begin{split}
        Y_{2}&\triangleq\,\sum_{i=1}^{n}\E\int_{\partial_{i}Q}\left(-2s\lambda\varphi(\gamma_i)^2\partial_{i}\psi+s\mathcal{O}_{\lambda}(sh)+h\mathcal{O}_{\lambda}(sh)\right)t_{r}^{i}(|D_{i}z|^{2})\nu_{i}\,dt\\
        &+\sum_{i=1}^{n}\E\int_{\partial_iQ}\mathcal{O}_{\lambda}(sh)t_r^{i}(|D_iz|^2)\,dt+\sum_{i,j=1}^{n}E\int_{\partial_{i}Q}s\mathcal{O}_{\lambda}((sh)^2)\,t_{r}^{i}(|D_{i}z|^{2})\nu_{i}\,dt.
    \end{split}
\end{equation*}
\end{lemma}

%%%%%%%%%%%%%%%%%%%%%%%%%%%%%%%%%%%%%%%%%%%%%%%%%%%%%%%%%%%%%%%%%%%%%%%%%%%%%%%%%%%%%%%%%%%%%%%%%%%%%%%%%%%%%%%%%%%%%%%%%%%%%%%%%%%%%%%%%%%%%%%%%%%%%%%%%%%%%%%%%%%%%%%%%%%%%%%%%%%%%%%%%%%%%%%%%%%%%%%%%

%%%%%%%%%%%%%%%%%%%%%%%%%%%%%%%%%%%%%%%%%%%%%%%%%%%%%%%%%%%%%%%%%%%%%%%%%%%%%%%%%%%%%%%%%%%%%%%%%%%%%%%%%%%%%%%%%%%%%%%%%%%%%%%%%%%%%%
Then by Lemma \ref{lem:dz}-\ref{lem:dt}, from \eqref{conmu} we obtain, for $0<\tau h (\max_{[0,T]}{\theta})\leq \varepsilon_1(\lambda)$
\begin{equation}\label{eq:sumterms}
\begin{split}
    2\E \int_Q C(z)B(z) &\geq\,-\sum_{j=1}^{3}X_{j}-\sum_{j=1}^2 Y_j+\E\int_{Q}2s^{3}\lambda^{4}\varphi^{3}|\nabla_{\gamma}\psi|^{4}\,|z|^{2}\,dt\\
    &+ \sum_{i=1}^{n}\E\int_{Q_{i}^{\ast}}2s\lambda^{2}\varphi\gamma_i|\nabla_{\gamma}\psi|^{2}\,|D_{i}z|^{2}\,dt+\sum_{i=1}^{n} \E\int_{Q^{\ast}}\gamma_i |D_i(dz)|^2.
    \end{split}
\end{equation}
To give an estimate of the right-hand side of \eqref{eq:Estirf}, we need the following estimation of $M_hz$, see \eqref{eq:opconjugate}. The proof can be adapted from Lemma 4.2 in \cite{boyer-2010-1d-elliptic} and the estimation of $\Phi$ in \cite{zhao:2024}.
\begin{lemma}\label{lem:estimate:termsMh}
(\textit{Estimate of $M_hz$.)} For $\tau h (\max_{[0,T]}{\theta})\leq 1$, we have
\begin{equation*}
    \E\int_{Q}|M_h(z)|^2\,dt\leq \mathcal{O}_{\lambda}(1)\left(\E\int_{Q}s^2|z|^2\,dt+h^2\sum_{i=1}^{n}\int_{Q^{\ast}}s^2|D_iz|^2\,dt\right).
\end{equation*}
\end{lemma}
Furthermore, we will need the following lemma, where its proof is analogous to Lemma 3.11 presented in \cite{boyer-2014}, and will therefore be omitted.
\begin{lemma}\label{lem:inequalityforDz}
    For $0<h\leq h_1(\lambda)$ for some $h_1(\lambda)$ sufficiently small, we have
    \begin{equation*}
         \begin{split}
        \sum_{i=1}^{n} \E\int_{Q_i^{\ast}} s\lambda^2\varphi \gamma_i |\nabla_{\gamma}\psi|^2|D_iz|^2\,dt\geq\frac{1}{n}\sum_{i=1}^{n}\E\int_{Q} s\lambda^2\varphi \gamma_i|\nabla_{\gamma}\psi|^2&|A_iD_iz|^2\,dt\\
        &+X_3+Y_3,
    \end{split}
\end{equation*}
 where 
 \begin{align*}
         &X_4\triangleq\, \sum_{\substack{i,j=1\\i\ne j}}^{n}\left(\E\int_{Q_{ij}^{\ast}}h\mathcal{O}_{\lambda}(sh)\,|D^2_{ij}z|^2\,dt-\E\int_{Q_{i}^{\ast}}h\mathcal{O}(sh)|D_iz|^2\,dt\right)\\
        &+\sum_{i=1}^{n}\left(\E\int_{Q} h\mathcal{O}_{\lambda}(sh)|A_iD_iz|^2\,dt+\E\int_{Q} h\mathcal{O}_{\lambda}(sh)\right.|D^2_iz|^2\,dt\\
        &\left.-\E\int_{Q_i^{\ast}} h\mathcal{O}_{\lambda}(sh)|D_iz|^2\,dt\right)
 \end{align*}
 and
 \begin{equation*}
     Y_3\triangleq\, \sum_{i=1}^{n}\E\int_{\partial_iQ} h\mathcal{O}(sh)t_r^i(|D_iz|^2)\,dt.
 \end{equation*}
\end{lemma}

Combining the lemma \ref{lem:estimate:termsMh} and \ref{lem:inequalityforDz} with \eqref{eq:sumterms}, we see that if we choose $\lambda_{1}\geq 1$ sufficiently large, then $\lambda=\lambda_1$(fixed for the rest of the proof), $0<\tau h(\max_{[0,T]\theta})\leq \varepsilon_{1}(\lambda)$ and $0<h\leq h_1(\lambda_1)$, we have
\begin{equation}\label{eq:conmutador}
\begin{split}
   \E\int_{Q}|rf|^2\,dt\geq &\,\E\int_{Q}2s^{3}\lambda^{4}\varphi^{3}|\nabla_{\gamma}\psi|^{4}\,|z|^{2}\,dt+ \sum_{i=1}^{n}\E\int_{Q_{i}^{\ast}}s\lambda^{2}\varphi|\nabla_{\gamma}\psi|^{2}\,|D_{i}z|^{2}\,dt\\
   &+\frac{1}{n}\sum_{i=1}^{n}\E\int_{Q} s\lambda^2\varphi \gamma_i|\nabla_{\gamma}\psi|^2|A_iD_iz|^2\,dt+\tilde{X}+\tilde{Y},
   \end{split}
\end{equation}
with
\begin{equation}
\begin{split}
    \tilde{X}\triangleq&\sum_{l=1}^{4}X_{l}+\sum_{m=1}^{3}Y_{m}-\E\int_{Q}s^{3}\lambda^{3}\varphi^{3}\mathcal{O}(1)|z|^{2}\,dt.
    \end{split}
\end{equation}
We can now choose $\varepsilon_{0}$ and $h_{0}$ sufficiently small, with $0<\varepsilon_{0} \leq \varepsilon_{1}(\lambda_{1})$, $0<h_0\leq h_1(\lambda_1)$, and $\tau_{1}\geq 1$ sufficiently large, such that for $\tau\geq \tau_1(T+T^2)$ (meaning in particular that $s(t)$ is taken sufficiently large), $0<h\leq h_{0}$, and $\tau h(\max_{[0,T]}\theta)\leq \varepsilon_{0}$, we obtain 
\begin{equation}\label{eq:firstestimate}
\begin{split}
    C_{s_{0},\varepsilon}\E\int_{Q}&|rf|^2\,dt\geq\,\E\int_{Q}s^{3}\lambda^{4}_{1}\varphi^{3}|\nabla_{\gamma}\psi|^{4}\,|z|^{2}+ \sum_{i=1}^{n}\E\int_{Q_{i}^{\ast}}s\lambda^{2}_{1}\varphi|\nabla_{\gamma}\psi|^{2}\,|D_{i}z|^{2}\\
    &+\frac{1}{n}\sum_{i=1}^{n}\E\int_{Q} s\lambda_{1}^2\varphi \gamma_i|\nabla_{\gamma}\psi|^2|A_iD_iz|^2\,dt-C_{s_0,\epsilon}\E\int_{Q}s^2|dz|^2-\tilde{Y},
    \end{split}
\end{equation}
where
\begin{equation}
    \begin{split}
        \tilde{Y}\triangleq \,\sum_{i=1}^{n}\E\int_{\partial_i Q }2s\lambda \varphi &(\gamma_i)^2\partial_i\psi\, t_r^i((D_iz|^2)\nu_i\,dt\\
        &+C_{s_0,\epsilon}\left(\left.\sum_{i=1}^{n}\E\int_{\mathcal{M}_i^{\ast}}|D_iz|^2\right|_0^T+\left.\E\int_{\mathcal{M}}s^2\,|z|^2\right|_0^T\right).
    \end{split}
\end{equation}
Moreover, since $|D_iz|^2\leq\,Ch^{-2}(|\tau_{-i}v|^2+|\tau_{+i}v|^2)$ and remember that $\psi$ satisfying \eqref{assumtion:psi}, we obtain
\begin{equation*}
    \tilde{Y}\leq\, C_{s_0,\epsilon}h^{-2}\left(\left.\E\int_{\mathcal{M}}|z|^2\right|_{t=0}+\left.\E\int_{\mathcal{M}}|z|^2\right|_{t=T}\right).
\end{equation*}
Therefore, we write \eqref{eq:firstestimate} as
\begin{equation}\label{eq:finalestimateinz}
\begin{split}  &\E\int_{Q}s^{3}\lambda^{4}_{1}\varphi^{3}\,|z|^{2}\,dt+ \sum_{i=1}^{n}\E\int_{Q_{i}^{\ast}}s\lambda^{2}_{1}\varphi\,|D_{i}z|^{2}\,dt+\sum_{i=1}^{n}\E\int_{Q} s\lambda_{1}^2\varphi|A_iD_iz|^2\,dt\\
&\leq\,\E\int_{0}^T\int_{G_{1}\cap \mathcal{M}}s^{3}\lambda^{4}_{1}\varphi^{3}\,|z|^{2}\,dt+ \sum_{i=1}^{n}\E\int_0^T\int_{G_{1}\cap\mathcal{M}_{i}^{\ast}}s\lambda^{2}_{1}\varphi\,|D_{i}z|^{2}\,dt\\ &+C_{s_{0},\varepsilon}\left(\E\int_{Q}|rf|^2\,dt+\E\int_{Q}s^2 \,|dz|^2+ h^{-2}\left.\E\int_{\mathcal{M}}|z|^2\right|_{t=0}+h^{-2}\left.\E\int_{\mathcal{M}}|z|^2\right|_{t=T}\right).
\end{split}
\end{equation}
%%%%%%%%%%%%%%%%%%%%%%%%%
\subsection{Returning back to the original variable}\label{sub:backVarible}
%%%%%%%%%%%%%%%%%%%%%%%%
Finally, we return to our original function. For this purpose, we obtain the following result. 
\begin{lemma}\label{lem:inequalityDz-Dw}
For $\tau h (\max_{[0,T]}{\theta})\leq 1$, we have for each $i=1,...,n$ 
\begin{equation*}
   \begin{split}
        \E\int_{Q_{i}^{\ast}}s\varphi \lambda^2&|rD_i w|^2\,dt\leq C\left(\E\int_{Q_i^{\ast}}s\varphi\lambda^2\,|D_iz|^2\,dt+\E\int_{Q_{i}^{\ast}}s\mathcal{O}_{\lambda}((sh)^2)\,|D_iz|^2\,dt\right.\\
        &\left. \E\int_{Q}s^3\varphi^3\lambda^4|z|^2\,dt+\E\int_{Q}s\mathcal{O}_{\lambda}((sh)^2)\,|z|^2\,dt+\E\int_{Q}s^3\mathcal{O}_{\lambda}((sh)^2)\,|z|^2\,dt\right),
    \end{split}
\end{equation*}
\begin{equation*}
    \begin{split}
        \E\int_0^T&\int_{G_{1}\cap\mathcal{M}_i^{\ast}}s\varphi\lambda^2|D_iz|^2\,dt\\
        &\leq\, C\left(\E\int_0^T\int_{G_{1}\cap\mathcal{M}_i^{\ast}}s\varphi\lambda^2\,|rD_iw|^2\,dt+\E\int_0^T\int_{G_{1}\cap \mathcal{M}_{i}^{\ast}}s\mathcal{O}_{\lambda}((sh)^2)\,|D_iz|^2\,dt\right.\\
        &\left. \E\int_0^T\int_{G_{1}\cap\mathcal{M}}s^3\varphi^3\lambda^4|z|^2\,dt+\E\int_0^T\int_{G_{1}\cap\mathcal{M}}s^3\mathcal{O}_{\lambda}((sh)^2)\,|z|^2\,dt\right)
    \end{split}
\end{equation*}
and
\begin{equation*}
    \begin{split}
        \E\int_{Q}s\lambda^2\varphi |rA_iD_iw|^2\,dt\leq \, C\left(\E\int_{Q}s\lambda^2\varphi|A_iD_iz|^2\,dt\right.&+\E\int_{Q}s^3\lambda^4\varphi\,|z|^2\,dt\\
        &\left.+\E\int_{Q}\mathcal{O}_{\lambda}((sh)^4)\,|D^2_iz|^2\,dt\right).
    \end{split}
\end{equation*}
\end{lemma}
For a proof of this result, it can be obtained using the same argument given in the final part of the proof of Theorem 1.4 in \cite{boyer-2014}, both for the discrete spatial derivative and for the average of this same operator. 

Combining the lemma \ref{lem:inequalityDz-Dw} with \eqref{eq:finalestimateinz} and noting
\begin{equation}\label{eq:equaldw^2}
    \E\int_{Q}s^2|dz|^2=\E\int_{Q}s^2|rdw|^2=\E\int_{Q}s^2|rg|^2\,dt,
\end{equation}
we conclude the following lemma.
\begin{lemma}\label{lem:last}
    Given any $\lambda>\lambda_1$, exits $\varepsilon_{0}$ and $h_{0}$ sufficiently small, with $0<\varepsilon \leq \varepsilon_0(\lambda_{1})$, and $\tau_{2}\geq \tau_{1}$ sufficiently large, such that for $\tau\geq \tau_{2}(T+T^2)$, $0<h\leq h_{0}$, and $\tau h(\max_{[0,T]}\theta)\leq \varepsilon_{0}$, 
\begin{equation*}
\begin{split}  
J(w)&+\E\int_{Q}s^{3}\lambda^{4}_{1}\varphi^{3}e^{2s\varphi}\,|w|^{2}\,dt \leq\,\E\int_{0}^T\int_{G_{1}\cap \mathcal{M}}s^{3}\lambda^{4}_{1}\varphi^{3}e^{2s\varphi}\,|w|^{2}\,dt\\
&+ \sum_{i=1}^{n}\E\int_0^T\int_{G_{1}\cap\mathcal{M}_{i}^{\ast}}s\lambda^{2}_{1}\varphi e^{2s\varphi}\,|D_{i}w|^{2}\,dt+C_{s_{0},\varepsilon}\left(\E\int_{Q}e^{2s\varphi}|f|^2\,dt\right.\\
&\left.+\E\int_{Q}s^2 e^{2s\varphi}\,|g|^2dt+ h^{-2}\left.\E\int_{\mathcal{M}}e^{2s\varphi}|w|^2\right|_{t=0}+h^{-2}\left.\E\int_{\mathcal{M}}e^{2s\varphi}|w|^2\right|_{t=T}\right).
\end{split}
\end{equation*}
where $J(w)\triangleq\, \sum_{i=1}^{n}\E\int_{Q_{i}^{\ast}}s\lambda^{2}_{1}\varphi e^{2s\varphi}\,|D_{i}w|^{2}\,dt+\E\int_{Q}s\lambda_{1}^2\varphi e^{2s\varphi}|A_iD_iw|^2\,dt$.
\end{lemma}
Now, we will need an estimate for the second term on the right-hand side in the previous lemma. For this purpose, we obtain the following Lemma.\begin{lemma}\label{lem:estimate:termDw}
    For $0<\tau h (\max_{[0,T]}\theta)\leq 1$, we have
    \begin{equation*}
        \begin{split}
            \sum_{i=1}^{n}\E\int_0^T\int_{G_{1}\cap\mathcal{M}_{i}^{\ast}}s\lambda^{2}_{1}\varphi e^{2s\theta\varphi}\,|D_{i}w|^{2}\,dt\leq \sum_{i=1}^{n}&\E\int_0^T\int_{G_{0}\cap\mathcal{M}_{i}^{\ast}}s\lambda^{2}_{1}\varphi e^{2s\theta\varphi}\,|D_{i}w|^{2}\,dt\\
            \leq C_{s_3,\epsilon}\left(\E\int_0^T\int_{G_{0}\cap\mathcal{M}}s^3\lambda^2\varphi ^3e^{2s\theta\varphi}\,|w|^2\,dt+\E\int_0^T\right.&\int_{G_{0}\cap\mathcal{M}}\lambda^{-2}e^{2s\theta\varphi}|f|^2\\
            &\left.+\left.h^{-1}\E\int_{G_{0}\cap\mathcal{M}}e^{2s\theta\varphi}\,|w|^2\right|_{t=T}\right).
        \end{split}
    \end{equation*}
\end{lemma} For a proof see Appendix \ref{sec:intermediateResults}.

By Lemma \ref{lem:last}-\ref{lem:estimate:termDw} and observing that since $\max_{[0,T]}\theta \leq \frac{1}{\delta T^2(1+\delta)}\leq \frac{1}{T^2 \delta}$, a sufficient condition for $\tau h(\max_{[0,T]}\theta) \leq \varepsilon_{0}$ become $\tau h(T^2\delta)^{-1}\leq \varepsilon_{0}$, then we complete the proof of the Theorem \ref{theo:Carleman}.

\end{proof}

%%%%%%%%%%%%%%%%%%%%%%%%%%%%%%%%%%%%%%%%%%%%%%%%%%%%%%%%%%%%%%%%%%%%%%%%%%%%%%%%%%%%%%%%%%%%%%%%%%%%%%%%%%%%%%%%%%%%%%%%%%%%%%%%%%%%%%%%%%%%%%%%%%%%%%%%%%%%%%%%%%%%%%%%%%%%%%%%%%%%%%%%%%%%%%%%%%%%%%%%%%%%%%%%%%%%%%%%%%%%%%%%%%%%%%%%%%%%%%%%%%%%%%%%%%%%%%%%%%%
\section{Observability and controllability of stochastic parabolic systems}\label{sec:ObservabilityandControllability}
This section is devoted to proving the $\phi$-observability inequality for semi-discrete stochastic parabolic equation \eqref{systemofcontrolDiscrete}, and then following the penalized HUM methodology to obtain the $\phi$-null controllability for the system \eqref{systemofcontrolDiscrete}. To perform this analysis, we will reduce our problem to establishing a $\phi$-observability inequality for $z$ solution of the following backward stochastic equation.
\begin{equation} \label{systemadjstohocastic}
\left\{
\begin{aligned}
    \mathcal{P}z 
    &= \left( \sum_{i=1}^n a_{1i} A_i D_i z - a_2 z - a_3 Z \right) dt 
    + Z\,dB(t) \text{ in}\, Q, \\
    z &= 0 \text{ on}\, \partial Q, \\
    z(T) &= z_T \text{ in}\, \mathcal{M},
\end{aligned}
\right.
\end{equation}
where $\mathcal{P}z=dz + \sum_{i=1}^n D_i(\gamma_i D_i z) dt $ and $z_T\in L^2_{\mathcal{F}_{T}}(\Omega;L_{h}^2(\mathcal{M}))$ is given. By duality, \eqref{systemadjstohocastic} is obtained from \eqref{systemofcontrolDiscrete} using Proposition \ref{prop:integralbyparts}.
\subsection{The observability for the backward equation.}
Applied the penalized HUM approach, the discrete Carleman estimate \eqref{eq:inequalityCarleman} implies a relaxed observability inequality for the system \eqref{systemadjstohocastic}, which is the following result.
\begin{theorem}\label{Teo:Observability inequality}
    There exists $h_0>0$ such that for all $h\leq h_0$, the following \textit{$\phi$-observability inequality} holds
    \begin{equation*}
        \E\int_{\mathcal{M}}|z(0)|\leq \, C\left(\E\int_{Q}|Z|^2\, dt+\E\int_0^T\int_{G_0\cap\mathcal{M}}|z|^2\,dt+\phi(h)\left.\E\int_{\mathcal{M}}|z|^2\right|_{t=T}\right)
    \end{equation*}
for any solution to the system \eqref{systemadjstohocastic} and $\phi$ satisfying \eqref{condphi}.
\end{theorem}
\begin{proof} First, using the Holder inequality, we obtain that
$$
\E\int_{Q}(e^{s\varphi}a_2z)^2
\; dt\leq C\E\int_0^T\left(\int_{\mathcal{M}}|a_2|^{2r}\right)^{\frac{2}{2r}}\left(\int_{\mathcal{M}}|e^{s\varphi}z|^{2q}\right)^{\frac{2}{2q}}\;dt,
$$
where $1/r+1/q=1$. Taking $n^{\ast}=2r$, using Theorem \ref{teo:discreteSobolev} with $p=2$ and $p^{\ast}=2q$, and the rule of product with respect to the differential operator, we have that $n^{\ast}$ satisfies \eqref{conditionofn*} and
\begin{align*}
    &\mathds{E}\int_{Q}(e^{s\varphi}a_2z)^2dt
    \leq C\|a_2\|^2_{L^{\infty}_{\mathds{F}}(0,T;L^{n^{\ast}}(\mathcal{M}))}\E\int_0^T\left(\sum_{i=1}^{n}\int_{\mathcal{M}_i^{\ast}}|D_i(e^{s\varphi}z)|^{2}+\int_{\mathcal{M}}|e^{s\varphi}z|^{2}\right)\,dt\\
    \leq& C\|a_2\|^2_{L^{\infty}_{\mathds{F}}(0,T;L^{n^{\ast}}(\mathcal{M}))}\E\int_0^T\left(\sum_{i=1}^{n}\int_{\mathcal{M}_i^{\ast}}|D_i(e^{s\varphi})|^2|A_iz|^2+|A_i(e^{s\varphi})|^2|D_iz|^2+\int_{\mathcal{M}}e^{2s\varphi}|z|^2\right)\,dt
\end{align*}
From \eqref{prop:weight}, we notice that $|D_i(e^{s\varphi})|^2\leq \lambda^2s^2\varphi^2|\nabla\psi|^2e^{2s\varphi}+e^{2s\varphi}\mathcal{O}((sh)^2)$ and $|A_i(e^{s\varphi})|^2\leq e^{2s\varphi}+e^{2s\varphi}\mathcal{O}((sh)^2)$, thus it follows that
\begin{align*}
    \mathds{E}\int_{Q}(e^{s\varphi}a_2z)^2dt
    &\leq C\|a_2\|^2_{L^{\infty}_{\mathds{F}}(0,T;L^{n^{\ast}}(\mathcal{M}))}\E\int_0^T\left(\sum_{i=1}^{n}\int_{\mathcal{M}_i^{\ast}}\lambda^2s^2\varphi^2e^{2s\varphi}|A_iz|^2+e^{2s\varphi}|D_iz|^2\right.\\
    &\left.\int_{\mathcal{M}}e^{2s\varphi}|z|^2\right)\,dt+\mathcal{O}_{\lambda}((sh)^2)\sum_{i=1}^{n}\mathds{E}\int_{Q_i^{\ast}}e^{2s\varphi}|A_iz|^2+e^{2s\varphi}|D_iz|^2\,dt
\end{align*}
We can observe that thanks to \eqref{eq:promedioinequality} and \eqref{eq:int:ave},
\begin{align*}
    \E\int_{Q_i^\ast}\varphi^2e^{2s\varphi}|A_iz|^2\,dt\leq \E\int_{Q_i^{\ast}}\varphi^2e^{2s\varphi}A_i|z|^2\,dt\leq \E\int_{Q}\varphi^2e^{2s\varphi}|z|^2\,dt+\E\int_{Q}\mathcal{O}((sh)^2)|z|^2\,dt
\end{align*}
where we used that $z=0$ on $\partial_i Q$. Therefore, we arrive at the following inequality
\begin{align*}
    \mathds{E}\int_{Q}(e^{s\varphi}a_2z)^2dt
    \leq& C\|a_2\|^2_{L^{\infty}_{\mathds{F}}(0,T;L^{n^{\ast}}(\mathcal{M}))}\mathds{E}\left(\int_{Q}(1+\lambda^2s^2\varphi^2)e^{2s\varphi}|z|^2\,dt+\sum_{i=1}^n\int_{Q_i^{\ast}}e^{2s\varphi}|D_iz|^2\,dt\right)\,\\
    &+\mathcal{O}_{\lambda}((sh)^2)\mathds{E}\left(\int_{Q}e^{2s\varphi}|z|^2\,dt+\sum_{i=1}^n\int_{Q_i^{\ast}}e^{2s\varphi}|D_iz|^2\,dt\right).
\end{align*}
Now, let
\begin{align*}
    \mathcal{C}_1\triangleq\,\sum_{i=1}^{n} \|a_{1i}\|_{L^{\infty}_{\mathds{F}}(0,T;W_{h}^{1,\infty}(\mathcal{M}))},\,\,\mathcal{C}_2\triangleq\,\|a_{2}\|_{L^{\infty}_{\mathds{F}}(0,T;L_{h}^{n^\ast}(\mathcal{M}))},\,\, \mathcal{C}_3\triangleq\,\|a_{3}\|_{L^{\infty}_{\mathds{F}}(0,T;L_{h}^{\infty}(\mathcal{M}))}.
\end{align*}
By Theorem \ref{theo:Carleman} and $z$ solution of the system \eqref{systemadjstohocastic}, it deduces that   
   \begin{equation*}
 \begin{split} 
 &\sum_{i=1}^{n}\E\int_{Q_{i}^{\ast}}s\varphi e^{2s\varphi}\,|D_{i}z|^{2}\,dt+\sum_{i=1}^{n}\E\int_{Q}s\varphi e^{2s\varphi}|A_iD_iz|^2\,dt+\E\int_{Q}s^{3}\varphi^{3}e^{2s\varphi}\,|z|^{2}\,dt\\
&\leq\,\mathcal{C}\left(\E\int_{0}^T\int_{G_{0}\cap \mathcal{M}}s^{3}\varphi^{3}e^{2s\varphi}\,|z|^{2}\,dt+\mathcal{C}_1^2\sum_{i=1}^{n}\E\int_{Q}e^{2s\varphi}|A_iD_i(z)|^2\,dt\right.\\
&+\left.\mathcal{C}_{2}^2\E\left(\int_{Q}(1+s^2\varphi^2)e^{2s\varphi}|z|^2\,dt+\sum_{i=1}^n\int_{Q_i^{\ast}}e^{2s\varphi}|D_iz|^2\,dt\right)+\E\int_{Q}e^{2s \varphi}(\mathcal{C}_3^2+s^2)|Z|^2\,dt\right.\\
&\left.+ h^{-2}\left.\E\int_{\mathcal{M}}e^{2s\varphi}|z|^2\right|_{t=0}+h^{-2}\left.\E\int_{\mathcal{M}}e^{2s\varphi}|z|^2\right|_{t=T}\right).
\end{split}
\end{equation*}
with $s=\tau\theta$ for $\tau\geq \tau_{0}(T+T^2)$, $0<h\leq h_{0}$ and $\tau h (\max_{[0,T]}{\theta})\leq 1$.

As $1\leq \mathcal{C}\theta T^2$ it suffices to have 
\begin{equation}\label{eq:parameterstau}
\tau\geq T^2(\mathcal{C}_1^{2}+\mathcal{C}_{2}^{2/3})
\end{equation}
to obtain
\begin{equation}\label{eq:secobservabilityP}
 \begin{split}
 \sum_{i=1}^{n}&\E\int_{Q_{i}^{\ast}}\tau\theta\varphi e^{2s\varphi}\,|D_{i}z|^{2}\,dt+\sum_{i=1}^{n}\E\int_{Q}\tau\theta\varphi e^{2s\varphi}|A_iD_iz|^2\,dt\\
 &+\E\int_{Q}\tau^3\theta^{3}\varphi^{3}e^{2s\varphi}\,|z|^{2}\,dt\leq\,\mathcal{C}\left(\E\int_{0}^T\int_{G_{0}\cap \mathcal{M}}\tau^{3}\theta^{3}\varphi^{3}e^{2s\varphi}\,|z|^{2}\,dt\right.\\
&+\E\int_{Q}\tau^{2}\theta^2e^{2s\varphi}|Z|^2\,dt+\left.+ h^{-2}\left.\E\int_{\mathcal{M}}e^{2s\varphi}|z|^2\right|_{t=0}+h^{-2}\left.\E\int_{\mathcal{M}}e^{2s\varphi}|z|^2\right|_{t=T}\right),
\end{split}
\end{equation}
where we choose $\tau_{1}\geq \tau_0$ sufficiently large to have \eqref{eq:parameterstau} for $\tau \geq \tau_{1}(T+T^2+T^2(\mathcal{C}_1^{2}+\mathcal{C}_{2}^{2/3}))$.

Recalling the definition of $\theta$, we have $\theta(t)\geq\theta(\frac{T}{2})\geq T^{-2}$, since $0<\delta<\frac{1}{2}$, and $\theta\leq \theta(\frac{T}{4})\leq \frac{16}{3T^2}$ for $t\in [\frac{T}{4},\frac{3T}{4}]$. Then
\begin{equation}\label{eq:secobservability1}
\begin{split}
    \E\int_{Q}\tau^3\theta^{3}\varphi^{3}e^{2s\varphi}\,|z|^{2}\,dt\geq &\E\int_{\frac{T}{4}}^{\frac{3T}{4}}\int_{\mathcal{M}}\tau^{3}\theta^{3}\varphi^{3}e^{2s\varphi}\,|z|^{2}\,dt\\
    \geq &\mathcal{C} e^{-\mathcal{C}\tau T^{-2}}\E\int_{\frac{T}{4}}^{\frac{3T}{4}}\int_{\mathcal{M}}\,|z|^{2}\,dt.
\end{split}
\end{equation}
On the other hand, notice that $d(|z|^2)=2zdz+|dz|^2=2zdz+Z^2dt$. Hence, for any $0\leq t_1 \leq t_2 \leq T$,
\begin{equation}\label{eq:estimateofz(0)}
    \begin{split}
        \E\int_{\mathcal{M}}|z(t_2)|^2-\E\int_{\mathcal{M}}|z(t_1)|^2=&\E\int_{t_1}^{t_2}\int_{\mathcal{M}}\{2z[-\sum_{i=1}^{n} D_i(\gamma_i D_iz)+\sum_{i=1}^{n}a_{1i}A_iD_i (z)\\
        &-a_2z-a_3Z]+|Z|^2\}\,dt.
    \end{split}
\end{equation}
By integration by parts \eqref{eq:int:dif} and notice that $z=0$ on $\partial_i \mathcal{M}$, we obtain that for each $i=1,...,n$
\begin{equation}\label{eq:estimateupDz}
\begin{split}
    \E\int_{t_1}^{t_2}\int_{\mathcal{M}} zD_i(\gamma_iD_iz)\,dt=&-\E\int_{t_1}^{t_2}\int_{\mathcal{M}_{i}^{\ast}}\gamma_i|D_iz|^2\,dt+\E\int_{t_1}^{t_2}\int_{\partial_{i}\mathcal{M}}zt_{r}^{i}(\gamma_iD_iz)\,dt\\
    =&-\E\int_{t_1}^{t_2}\int_{\mathcal{M}_{i}^{\ast}}\gamma_i|D_iz|^2\,dt\leq 0.
\end{split}
\end{equation}
In turn, let us focus on the term $\displaystyle\mathcal{I}=\E\int_{t_1}^{t_2}\int_{\mathcal{M}} a_{1i}z\,A_iD_iz\,dt$. We note that using \eqref{eq:int:dif}, that $z=0$ on $\partial \mathcal{M}$, and then \eqref{eq:difference:product}, we have
\begin{equation*}
\begin{split}
    \mathcal{I}=&-\E\int_{t_1}^{t_2}\int_{\mathcal{M}_{i}^{\ast}}D_i(a_{1i}z)A_iz\,dt+\E\int_{t_1}^{t_2}\int_{\partial_{i}\mathcal{M}}a_{1i}z\,t_{r}^{i}(A_iz)\nu_{i}\,dt\\
     =&-\E\int_{t_1}^{t_2}\int_{\mathcal{M}_{i}^{\ast}}D_i(a_{1i})|A_iz|^2\,dt-\frac{1}{2}\E\int_{t_1}^{t_2}\int_{\mathcal{M}_{i}^{\ast}}A_i(a_{1i})D_{i}(|z|^{2})\,dt,\\
\end{split}
\end{equation*}
where the last integral above is rewritten according \eqref{eq:derivadacuadrado}. Furthermore, applying \eqref{eq:promedioinequality} on the first integral above and integration for parts with respect to difference operator for the second integral, we see that for each $i=1,...,n$
\begin{equation}\label{eq:estiamteupDAz}
\begin{split}
     \mathcal{I}
     \geq& -\E\int_{t_1}^{t_2}\int_{\mathcal{M}_{i}^{\ast}}D_i(a_{1i})A_i(|z|^2)\,dt+\frac{1}{2}\E\int_{t_1}^{t_2}\int_{\mathcal{M}}D_{i}A_i(a_{1i})|z|^2\,dt\\
     &-\E\int_{\partial_i \mathcal{M}}|z|^2t_{r}^{i}(A_ia_{1i})\nu_{i}\,dt\\
     %\geq& -\frac{1}{2}\E\int_{t_1}^{t_2}\int_{\mathcal{M}}D_iA_i(a_{1i})\,|z|^2\,dt+\frac{h}{2}\int_{\partial_i\mathcal{M}}|z|^2t_{r}^{i}(D_ia_{1i})\,dt\\
     \geq&  -\frac{1}{2}\E\int_{t_1}^{t_2}\int_{\mathcal{M}}D_iA_i(a_{1i})\,|z|^2\,dt,
\end{split}
\end{equation}
where the last line is obtained after an integration by parts for the average operator and we used that $z=0$ on $\partial_i Q$. Therefore, replacing \eqref{eq:estimateupDz} and \eqref{eq:estiamteupDAz} on \eqref{eq:estimateofz(0)}, and using Young's inequality, it holds that
\begin{equation*}
    \begin{split}
        \E&\int_{\mathcal{M}}|z(t_2)|^2\,dt-\E\int_{\mathcal{M}}|z(t_1)|^2\,dt\\
        \geq&\E\int_{t_1}^{t_2}\int_{\mathcal{M}}\left[-\left(\sum_{i=1}^{n}D_iA_i(a_{1i})+a_2\right)\,|z|^2-2a_3zZ+Z^2\right]\,dt\\
        \geq&-\left(\mathcal{C}_1+\mathcal{C}_2+\frac{\mathcal{C}_3^2}{2}\right)\E\int_{t_1}^{t_2}\int_{\mathcal{M}}|z|^2\,dt.
    \end{split}
\end{equation*}
By the Gronwall inequality, it holds that
\begin{equation}\label{secinequalityObseq:1}
    \E\int_{\mathcal{M}}|z(t_1)|^2\,dt\leq\, \mathcal{C}\E \int_{\mathcal{M}}|z(t_2)|^2\,dt,\quad\quad0,\leq t_1\leq t_2 \leq T.
\end{equation}
Thus, using \eqref{secinequalityObseq:1} on \eqref{eq:secobservability1}, it follows that
\begin{equation}\label{eq:secobservability2}
    \E\int_{Q}\tau^3\theta^3\varphi^3e^{2s\varphi}|z|^2\,dt\geq\, \mathcal{C}e^{-\mathcal{C}\frac{\tau}{T^2}}\E\int_{\mathcal{M}}|z(0)|^2\,dt.
\end{equation}
Noting that $\theta(T)=\theta(0)\geq \frac{2}{3}(\delta T^2)^{-1}$ and using the inequality \eqref{secinequalityObseq:1}, then
\begin{equation}\label{eq:secobservability3}
    \begin{split}
    h^{-2}\E&\left.\int_{\mathcal{M}}e^{2s\varphi}|z|^2\right|_{t=0}+h^{-2}\left.\E\int_{\mathcal{M}}e^{2s\varphi}|z|^2\right|_{t=T}\\
        \leq& \mathcal{C}h^{-2}e^{-\frac{\mathcal{C}}{\delta T^2}}\left[\left.\E\int_{\mathcal{M}}|z|^2\right|_{t=0}+\left.\E\int_{\mathcal{M}}|z|^2\right|_{t=T}\right]\leq  \mathcal{C}h^{-2}e^{-\frac{\mathcal{C}}{\delta T^2}}\left.\E\int_{\mathcal{M}}|z|^2\right|_{t=T}.
    \end{split}
\end{equation}
Moreover, $\theta(t)\geq \theta(\frac{T}{2})\geq T^{-2}$, we conclude that
\begin{equation}\label{eq:secobservability4}
\begin{split}
    \E\int_{G_{0}\cap\mathcal{M}}\theta^3e^{2s\varphi}z^2\, dt&+\E\int_{Q}\theta^2e^{2s\varphi}|Z|^2\,dt\\
    \leq & e^{-\frac{\mathcal{C}}{T^2}}\left[ \E\int_{G_{0}\cap\mathcal{M}}z^2\, dt+\E\int_{Q}|Z|^2\,dt\right].
\end{split}
\end{equation}
Combining \eqref{eq:secobservabilityP} and \eqref{eq:secobservability2}, we have
\begin{equation*}
    \begin{split}
        \E\int_{\mathcal{M}}&|z|^2(0)\leq \mathcal{C}e^{\mathcal{C}\tau T^{-2}}\left(\E\int_{0}^T\int_{G_{0}\cap \mathcal{M}}\theta^{3}\varphi^{3}e^{2s\varphi}\,|z|^{2}\,dt+\E\int_{Q}\theta^2e^{2s\varphi}|Z|^2\,dt\right.\\
&\left.+ h^{-2}\left.\E\int_{\mathcal{M}}e^{2s\varphi}|z|^2\right|_{t=0}+h^{-2}\left.\E\int_{\mathcal{M}}e^{2s\varphi}|z|^2\right|_{t=T}\right).
\end{split}
\end{equation*}
Moreover, using \eqref{eq:secobservability3} and \eqref{eq:secobservability4}, we see that
\begin{equation*}
\begin{split}
\E\int_{\mathcal{M}}&|z|^2(0)\\
\leq & \mathcal{C} e^{\mathcal{C}\frac{\tau}{T^2}}\left(\E\int_{0}^T\int_{G_{0}\cap \mathcal{M}}\,|z|^{2}\,dt+\E\int_{Q}|Z|^2\,dt\right)+h^{-2}e^{-\frac{\mathcal{C}}{\delta T^{2}}+\frac{\mathcal{C}}{T^2}}\left.\E\int_{\mathcal{M}}|z|^2\right|_{t=T}.
    \end{split}
\end{equation*}
For $0<\delta<\delta_{1}\leq \delta_0$, where $\delta_1$ is sufficiently small, we obtain 
\begin{equation}\label{eq:fiobser}
\begin{split}
    \E\int_{\mathcal{M}}|z|^2(0)\leq \mathcal{C} e^{\mathcal{C}\frac{\tau}{T^2}}&\left(\E\int_{0}^T\int_{G_{0}\cap \mathcal{M}}\,|z|^{2}\,dt+\E\int_{Q}|Z|^2\,dt\right)\\
    &+h^{-2}e^{-\frac{\mathcal{C}}{\delta T^{2}}}\left.\E\int_{\mathcal{M}}|z|^2\right|_{t=T}.
\end{split}
\end{equation}

We recall the conditions of Theorem \ref{theo:Carleman}
\begin{align}
    \frac{\tau h}{\delta T^2}\leq \varepsilon_0,\quad&\quad h\leq h_0.
\end{align}
They need to be fulfilled along with $\delta\leq \delta_1$.
We fix $\tau=\tau_0(T+T^2+T^2(\mathcal{C}_3+\mathcal{C}_1^2+\mathcal{C}_4^{2}+\mathcal{C}_2^{2/3}))$ with $\tau_0$ as chosen in Theorem \ref{theo:Carleman}. We define $h_1$ through 
\begin{equation*}
    h_{1}=\frac{\varepsilon_0}{\tau_0}\delta_{1}(1+\frac{1}{T}+\mathcal{C}_1^2+\mathcal{C}_4^{2}+\mathcal{C}_2^{2/3})^{-1},
\end{equation*}
which gives $\frac{\tau h_1}{\delta_{1}T^2}=\varepsilon_{0}$. We choose $h\leq \min{(h_0,h_1)}$, and $\delta=h\delta_1/h_1\leq \delta_{1}$ we then find $\frac{\tau h}{\delta T^2}=\varepsilon_{0}$.

As $\tau/(T^2\delta)=\epsilon_0/h$, we obtain from \eqref{eq:fiobser}
\begin{equation*}
     \E\int_{\mathcal{M}}|z|^2(0)\leq  \mathcal{C}\left(\E\int_{0}^T\int_{G_{0}\cap \mathcal{M}}\,|z|^{2}\,dt+\E\int_{Q}|Z|^2\,dt\right)+h^{-2}e^{-\frac{\mathcal{C}\varepsilon_0}{h}}\left.\E\int_{\mathcal{M}}|z|^2\right|_{t=T}.
\end{equation*}
This completes the proof of Theorem \ref{Teo:Observability inequality}.
\end{proof}
\subsection{The controllability for the stochastic semi-discrete parabolic equation}
Finally, our third and last main result establishes the \textit{$\phi$-null controllability} property for system \eqref{systemofcontrolDiscrete}.
\begin{theorem}\label{theo:nullcontrolability}
    There exist $C$ and $h_0$, for all $h\leq h_0$, there exist $(u,v)\in L^2_{\mathds{F}}(0,T;L^2(G_{0}\cap\mathcal{M}))\times L^2_{\mathds{F}}(0,T;L^2(\mathcal{M}))$ such that the solution to $\eqref{systemofcontrolDiscrete}$ satisfies
    \begin{equation*}
        \E\int_{Q}|v|^2\,dt+\E\int_0^T\int_{G_0\cap \mathcal{M}}|u|^2\,dt\leq C \E\int_{\mathcal{M}}|y_0|^2,
    \end{equation*}
and
\begin{equation*}
    \E\int_{\mathcal{M}}|y(T)|^2\leq C \phi(h)\E\int_{\mathcal{M}}|y_0|^2.
\end{equation*}
for any $\phi$ satisfying \eqref{condphi}.
\end{theorem}

\begin{proof} For any $y_{0} \in L_{\mathcal{F}_0}^{2}(\Omega;L^2_{h}(\mathcal{M}))$ and $\phi(h)$ that satisfies \eqref{condphi}, define a functional $J_{\varepsilon}$ on $L_{\mathcal{F}_{T}}^{2}\left(\Omega; L_{h}^{2}(\mathcal{M})\right)$ for any value of the penalization parameter $\varepsilon>0$ as follows:
$$
J_{\varepsilon}\left(z_{T}\right)\triangleq\frac{1}{2} \E \int_{Q}|Z|^{2} d t+\frac{1}{2} \E \int_{0}^{T} \int_{G_{0} \cap \mathcal{M}}|z|^{2} d t+\frac{\phi(h)}{2} \E\left\|z_{T}\right\|_{L_{h}^{2}(\mathcal{M})}^{2}-\E\left\langle y_{0}, z(0)\right\rangle_{L_{h}^{2}(\mathcal{M})}
$$
where $\varepsilon=\phi(h)$, and $(z, Z)$ is the solution to \eqref{systemadjstohocastic} with final value $z_{T}$. It is a simple matter to see that the functional $J(\cdot)$ is continuous and convex. Now we prove that it is coercive. By Young's inequality and Theorem \ref{Teo:Observability inequality}, we have
$$
\begin{aligned}
J_{\varepsilon}\left(z_{T}\right) \geq & \frac{1}{2} \E \int_{Q}|Z|^{2} d t+\frac{1}{2} \E \int_{0}^{T} \int_{G_{0} \cap \mathcal{M}}|z|^{2} d t+\frac{\phi(h)}{2}\left\|z_{T}\right\|_{L_{h}^{2}(\mathcal{M})}^{2} \\
& -\frac{1}{\epsilon}\|z(0)\|_{L_{h}^{2}(\mathcal{M})}^{2}-\
\epsilon\left\|y_{0}\right\|_{L_{h}^{2}(\mathcal{M})}^{2} \\
\geq & \frac{1}{4} \E \int_{Q}|Z|^{2} d t+\frac{1}{4} \E \int_{0}^{T} \int_{G_{0} \cap \mathcal{M}}|z|^{2} d t+\frac{\phi(h)}{4}\left\|z_{T}\right\|_{L_{h}^{2}(\mathcal{M})}^{2}-4C\left\|y_{0}\right\|_{L_{h}^{2}(\mathcal{M})}^{2}
\end{aligned}
$$
where $\epsilon=4C$ with $C$ a constant given by Theorem \ref{Teo:Observability inequality}. Then, $J(\cdot)$ is coercive. Therefore, $J_{\varepsilon}(\cdot)$ admits a unique minimizer $z_{T}^{*}$. Denote $\left(z^{*}, Z^{*}\right)$ the solution to \eqref{systemadjstohocastic} with final value $z_{T}^{*}$. Then we have
$$
\E \int_{Q} Z Z^{*} d t+\E \int_{0}^{T} \int_{G_{0} \cap \mathcal{M}} z z^{*} d t+\phi(h) \E\left\langle z_{T}, z_{T}^{*}\right\rangle_{L_{h}^{2}(\mathcal{M})}-\E\left\langle y_{0}, z(0)\right\rangle_{L_{h}^{2}(\mathcal{M})}=0
$$
for any $z_{T}$, with the associated solution $(z, Z)$ to \eqref{systemadjstohocastic}. We set $u^{*}=-z^{*} \chi_{G_{0} \cap \mathcal{M}}$ and $v^{*}=-Z^{*}$. We claim that \eqref{systemadjstohocastic} with $u=u^{*}$ and $v=v^{*}$ satisfy our needs. By the duality of \eqref{systemadjstohocastic} with \eqref{systemofcontrolDiscrete}, we have
$$
\E \int_{\mathcal{M}} y(T) z_{T}-\E \int_{\mathcal{M}} y_{0} z(0)=\E \int_{Q} v Z d t+\E \int_{0}^{T} \int_{G_{0} \cap \mathcal{M}} z u d t,
$$
then $y(T)=\phi(h) z_{T}^{*}$. On the other hand,
\begin{equation}\label{eq:proofcont1}
\begin{split}
& \E \int_{Q}\left|Z^{*}\right|^{2} d t+\E \int_{0}^{T} \int_{G_{0} \cap \mathcal{M}}\left|z^{*}\right|^{2} d t+\phi(h) \E\left\|z_{T}^{*}\right\|_{L_{h}^{2}(\mathcal{M})}^{2}=\E\left\langle y_{0}, z^{*}(0)\right\rangle_{L_{h}^{2}(\mathcal{M})} \\
& \leq \frac{1}{4 \epsilon} \E \int_{\mathcal{M}}\left|y_{0}\right|^{2} d t+\epsilon \E \int_{\mathcal{M}}\left|z^{*}(0)\right|^{2} d t.
\end{split}
\end{equation}
Notice that
\begin{equation}\label{eq:proofcont2}
\E \int_{\mathcal{M}}\left|z^{*}(0)\right|^{2} d t \leq C\left(\E \int_{Q}\left|Z^{*}\right|^{2} d t+\E \int_{0}^{T} \int_{G_{0} \cap \mathcal{M}}\left|z^{*}\right|^{2} d t+\left.\phi(h) \E \int_{\mathcal{M}}\left|z^{*}\right|^{2}\right|_{t=T}\right).
\end{equation}
Using \eqref{eq:proofcont1} and \eqref{eq:proofcont2}, we have that
$$
\E \int_{Q}\left|Z^{*}\right|^{2} d t+\E \int_{0}^{T} \int_{G_{0} \cap \mathcal{M}}\left|z^{*}\right|^{2} d t \leq C \E \int_{\mathcal{M}}\left|y_{0}\right|^{2}
$$
and
$$
\phi(h)  \E\left\|z_{T}^{*}\right\|_{L_{h}^{2}(\mathcal{M})}^{2} \leq C \E \int_{\mathcal{M}}\left|y_{0}\right|^{2}.
$$

Hence
$$
\E \int_{Q}|v|^{2} d t+\E \int_{0}^{T} \int_{G_{0} \cap \mathcal{M}}|u|^{2} d t \leq C \E \int_{\mathcal{M}}\left|y_{0}\right|^{2}
$$
and
$$
\E\int_{\mathcal{M}}|y(T)|^2 \leq C \phi(h) \E \int_{\mathcal{M}}\left|y_{0}\right|^{2}.
$$
Then, we complete the proof.
\end{proof}
%%%%%%%%%%%%%%%%%%%%%%%%%%%%%%%%%%%%%%%%%%%%%%%%%%%%%%%%%%%%%%%%%%%%%%%%%%%%%%%%%%%%%%%%%%%%%%%%%%%%%%%%%%%%%%%%%%
\section{Concluding remarks}
We have established controllability results for the linear semi-discrete stochastic parabolic operator by using two controls. Unlike in the continuous setting, we have shown that null controllability cannot be expected for any initial data. In this context, we have proved the existence of bounded semi-discrete controls such that the norm of the solution approaches zero at a given time as the discrete parameter tends to zero. It is known that null-controllability is achieved using one control when the coefficients of the system only depend on the time variable \cite{Qi.Lu:2011}, and it could be interesting to explore a similar setting for the semi-discrete case.

\par In addition, the results in Theorem \ref{theo:nullcontrolability} can be extended to the more general case when we replace the first equation of \eqref{systemofcontrolDiscrete} by
\begin{equation*}
\tilde{\mathcal{P}}_h y = \left(\sum_{i=1}^{n} A_i D_i(a_{1i} y) + a_2 y + \chi_{G_0 \cap \mathcal{M}} u\right) \, dt + \sum_{k=1}^{m} \left(a_3^k y + v\right) \, dB_k(t),
\end{equation*}
where $a_3^k$ are functions in $L_{\mathds{F}}^{\infty}(0,T;L^{\infty}_{h}(\mathcal{M}))$ and $\{B_k(t)\}_{k=1,\cdots,m}$ are independent standard one-dimensional Brownian motions. This extension does not introduce significant analytical complexity, since the multidimensional nature of the Brownian motion only affects the notation through additional indices. The structure of Itô's formula remains essentially unchanged, and the stochastic integrals preserve their standard properties under the assumption of independence. Consequently, the Carleman inequality can still be derived in this setting, leading to the corresponding observability inequality for the adjoint system. Therefore, for the sake of clarity and without loss of generality, we restrict our analysis to the case $m = 1$. 
\par On the other hand, we note that the nonlinear semi-discrete case remains an open problem, particularly in the multidimensional setting and when the nonlinearity depends on both the state and its gradient. A significant advance has recently been made in the one-dimensional case with nonlinearity depending only on the state, as shown in \cite{wang2025nullcontrollabilitysemidiscretestochastic}, where the authors developed suitable Carleman estimates to address the $\phi$-null controllability of the nonlinear semi-discrete system.
In contrast, a more general class of nonlinearities—depending on both the state and its gradient—has been studied in the continuous setting in \cite{zhang2024newglobalcarlemanestimates}. In both works, the main tool is the derivation of appropriate Carleman estimates for the associated linearized systems. These results suggest that adapting this methodology to the semi-discrete setting could be a promising direction, although new challenges arise, especially due to the interaction between spatial discretization and time-dependent nonlinearities.

In \cite{baroun2023nullcontrollabilitystochasticparabolic}, a null controllability result is obtained for a similar system like \eqref{systemofcontrolcontinuos} although with the general convection term “$[a_5+B\cdot \nabla y]dB(t)$”. Then it could be interesting to analyze the semi-discrete behavior of the controls in that setting since the main tool used in that work is a new Carleman estimate. 
\par  Related to ongoing work is the consideration of the semi-discrete inverse problem for the stochastic parabolic operators; see, for instance, \cite{MR4729334} for the continuous setting and the references therein. Recently, Bin Wu, Ying Wang, and Zewen Wang published results on semi-discrete stochastic inverse problems in the one-dimensional setting \cite{Wu_2024}. Hence, the methodology developed in the proof of the Carleman estimate from Theorem \ref{theo:Carleman} could be adapted to extend the results from \cite{Wu_2024} to arbitrary dimensions, as their approach to obtaining the Carleman estimate follows a similar strategy. Moreover, the fundamental steps rely on the estimate of the discrete spatial operator applied on the respective Carleman weight function, and no significant difficulties are related with the temporal variable.
\section*{Acknowledgments}
This work is partially supported by the Math-Amsud project
CIPIF 22-MATH-01. R. Lecaros was partially supported by FONDECYT (Chile) Grant 1221892. A. A. P\'erez acknowledges the support of Vicerrector\'ia de Investigaci\'on y postgrado, Universidad del B\'io-B\'io, proyect IN2450902 and FONDECYT Grant 11250805. M. F. Prado gratefully acknowledges the support from the Institutional Scholarship Fund of the University of Valparaíso (FIB-UV).
\section*{Statements and Declarations}
The authors have no relevant financial or non-financial interests to disclose.
%%%%%%%%%%%%%%%%%%%%%%%%%%%%%%%%%%%%%%%%%%%%%%%%%%%%%%%%%%%%%%%%%%%%%%%%%%
\begin{appendices}

\section{Calculus results related to the weight functions}\label{sec:weight:function}
This section is devoted to present the results on discrete operations performed on the Carleman weight functions proved in \cite{AA:perez:2024} and \cite{boyer-2010-elliptic}, for continuous and discrete operators in arbitrary dimensions; respectively. Recall that our Carleman weight function is of the form $r\triangleq e^{s(t)\varphi(x)}$, for $s\geq 1$, with $\varphi(x)=e^{\lambda\psi(x)}-e^{\lambda K}<0$, where $\psi\in C^{k}$ for $k$ sufficiently large and $\lambda\geq 1$, $K>\left\| \psi\right\|_{\infty}$ and $s(t)\triangleq \tau\theta(t)$ for $\theta(t)\triangleq \left( (t+\delta T)(T+\delta T-t)\right)^{-1}$. We denote $\rho\triangleq r^{-1}$. Now, we consider two fundamental estimates for our weight function. 
\begin{lemma}[{\cite[Lemma 3.7]{boyer-2010-elliptic}}] \label{lem:derivative:wrt:x}
Let $\alpha$ and $\beta$ multi-indices. We have
\begin{equation}
\begin{split}
\partial^{\beta}(r\partial^{\alpha} \rho)=&|\alpha|^{|\beta|}(-s\varphi)^{|\alpha|}\lambda^{|\alpha+\beta|}(\nabla \psi)^{\alpha+\beta}+|\alpha||\beta|(s\varphi)^{|\alpha|}\lambda^{|\alpha+\beta|-1}\mathcal{O}(1)\\
&+s^{|\alpha|-1}|\alpha|(|\alpha|-1)\mathcal{O}_{\lambda}(1)=\mathcal{O}_{\lambda}(s^{|\alpha|}).
\end{split}
\end{equation}
Let $\sigma\in [-1,1]$ and $i=1,...,n$. We have
\begin{equation}
\begin{split}
\partial^{\beta}\left( r(x)(\partial^{\alpha}\rho)(x+\sigma h_{i})\right)=\mathcal{O}_{\lambda,\varepsilon}(s^{|\alpha|}(1+(sh)^{|\beta|}))e^{\mathcal{O}_{\lambda}(sh)}.
\end{split}
\end{equation}
Provided $sh\leq \varepsilon$ we have $\partial^{\beta}\left( r(x)(\partial^{\alpha}\rho)(x+\sigma h_{i})\right)=\mathcal{O}_{\lambda,\varepsilon}(s^{|\alpha|})$. The same expressions hold with $r$ and $\rho$ interchanged and with $s$ changed into $-s$.
\end{lemma}
Moreover, we also have
\begin{corollary}[{\cite[Corollary 3.8]{boyer-2010-elliptic}}]\label{Cor:derivative wrt x  for 2r}
Let $\alpha,\beta$ and $\delta$ multi-indices. We have
\begin{equation}
\begin{split}
\partial^{\delta}\left(r^{2}\,\partial^{\alpha}\rho\, \partial^{\beta}\rho\right)=&|\alpha+\beta|^{|\delta|}(-s\varphi)^{|\alpha+\beta|}\lambda^{|\alpha+\beta+\delta|}(\nabla\psi)^{\alpha+\beta+\delta}\\
&+|\delta||\alpha+\beta|(s\varphi)^{|\alpha+\beta|}\lambda^{|\alpha+\beta+\delta|-1}\mathcal{O}(1)\\
&+s^{|\alpha+\beta|-1}\left( |\alpha|(|\alpha|-1)+|\beta|(|\beta|-1)\right)\mathcal{O}(1)\\
&=\mathcal{O}_{\lambda}(s^{|\alpha+\beta|}).
\end{split}
\end{equation}
\end{corollary}
Now, for completeness, we state two asymptotic behavior of the Carleman weight function considered in this work. The proofs of these results can be found in \cite{AA:perez:2024}. We adopt a similar notation for multi-indices for space semi-discrete operator. We define for the two-dimensional index $k=(k_{i},k_{j})$ with $k_{i},k_{j}\in\mathbb{N}$ the discrete operator $\mathbf{D}_{h}^{k}\triangleq D_{i}^{k_{i}}D_{j}^{k_{j}}$. Similarly, for the average operator, we define $\mathbf{A}_{h}^{l}\triangleq A_{i}^{l_{i}}A_{j}^{l_{j}}$ with $l\triangleq (l_{i},l_{j})$ being a two-dimensional multi-index. Considering the aforementioned notation the next two results gives us an estimate of several computations of the discrete average and derivative operators applied to our Carleman weight function. These estimates will be fundamental in the proof of the Carleman estimate from Theorem \ref{theo:Carleman}. 
\begin{proposition}[{\cite[Proposition 3.3]{AA:perez:2024}}]\label{prop:weight}
    Let $\alpha$ a multi-index and $k,l$ two-dimensional indices, respectively. Let $i,j=1,...,n$; provided $sh\leq \varepsilon$, we have
    \begin{align}
r\mathbf{A}_{h}^{l}\mathbf{D}_{h}^{k}\partial^{\alpha}\rho=r\partial^{k+\alpha}\rho+&s^{|\alpha|+k_{i}}\mathcal{O}_{\lambda}((sh)^{2})\label{eq:average:difference:weight}\\
&+s^{|\alpha|+k_{j}}\mathcal{O}_{\lambda}((sh)^{2})+s^{|\alpha+k|}\mathcal{O}_{\lambda}((sh)^{2})\notag.
    \end{align}
\end{proposition}
%%%%%%%%%%%%%%%%%%%%%%
%%%%%%%%%%%%%%%%%%%%%%
\begin{comment}
Another useful estimate is the following lemma
\begin{lemma}\label{lem:difference:average:2}
Let $\alpha$, $\beta$ and $\delta$ be multi-index and let $k,l$ be 2-dimensional indexes. Provided $sh\leq 1$, we have
\begin{equation}\label{lem:difference:average}
\mathbf{A}_{h}^{l}\mathbf{D}_{h}^{k}\left(\partial^{\beta}\left( r\partial^{\alpha}\rho\right) \right)=\partial_{i}^{k_{i}}\partial_{j}^{k_{j}}\partial^{\beta}\left(r\partial^{\alpha}\rho \right)+h^{2}\mathcal{O}_{\lambda}(s^{|\alpha|}).
\end{equation}
Moreover, for $i\neq j=1,...,n$ we have
\begin{equation}\label{eq0:lem:difference:average}
\mathbf{A}_{h}^{l}\mathbf{D}_{h}^{k}\partial^{\delta}\left(r^{2}\left(\partial^{\alpha}\rho\right)\partial^{\beta}\rho\right)=\partial^{k+\delta}\left(r^{2}(\partial^{\alpha}\rho)\partial^{\beta}\rho \right)+h^{2}\mathcal{O}(s^{|\alpha|+|\beta|}),
\end{equation}
and for $\sigma,\sigma'\in[0,1]$ we get
\begin{align}
\mathbf{A}_{h}^{l}\mathbf{D}^{k}_{h}\partial^{\beta}\left(r(x)\partial^{\alpha}\rho(x+\sigma he_{i
} \right)&=\mathcal{O}_{\lambda}(s^{|\alpha|}),\label{eq1:lem:difference:average}\\
\mathbf{A}_{h}^{l}\mathbf{D}_{h}^{k}\partial^{\delta}\left(r^{2}(x)\left( \partial^{\alpha}\rho(x+\sigma h)\right)\partial^{\beta}\rho(x+\sigma' h)\right)&=\mathcal{O}_{\lambda}(s^{|\alpha|+|\beta|}).\label{eq2:lem:difference:average}
\end{align}
The same expressions hold with $r$ and $\rho$ interchanged.
\end{lemma}
\end{comment}
%%%%%%%%%%%%%%%%%%%%%%%%%%%
%%%%%%%%%%%%%%%%%%%%%%%%%%%
\begin{theorem}[{\cite[Theorem 3.5]{AA:perez:2024}}]\label{theo:weight:estimates}
Let $\alpha$ be a multi-index and $c,l,k,m,p$ and $q$ two-dimensional multi-indices. For $i,j=1,...n$, and for $sh\leq 1$, we have
\begin{equation}\label{eq1:theo:weight:estimates}
       \mathbf{A}_{h}^{l}\mathbf{D}_{h}^{k}\partial^{\alpha}(r\mathbf{A}_{h}^{m}\mathbf{D}_{h}^{c}\rho)=\partial^{k+\alpha}(r\partial^{c}\rho)+s^{|c|}\mathcal{O}_{\lambda}((sh)^{2})
\end{equation}
and 
\begin{equation}\label{eq2:theo:weight:estimates}
\mathbf{A}_{h}^{l}\mathbf{D}_{h}^{k}\partial^{\beta}(r^{2}\mathbf{A}^{p}_{h}\mathbf{D}^{q}_{h}(\partial^{\alpha}\rho)\mathbf{A}_{h}^{m}\mathbf{D}_{h}^{c}\rho)=\partial^{k+\beta}(r^{2}(\partial^{q+\alpha}\rho)\partial^{c}\rho)+s^{|\alpha+c+q|}\mathcal{O}_{\lambda}((sh)^{2}).
\end{equation}
\end{theorem}
Proposition \ref{prop:weight} and Theorem \ref{theo:weight:estimates} are extension into higher dimension of the  asymptotic behavior of Carleman weight functions obtained in \cite{LOP-2020}.
%%%%%%%%%%%%%%%%%%%%%%%%%%%%%%%%%%%%%%%%%%%%%%%%%%%%%%%%%%%%%%%%%%%%%%%%%%%%%%%%%%%%%%%%%%%%%%%%%%%%%%%%%%%%%%%%%%%%%%%%%%%%%%%%%%%%%%%%%%%%%%%%%%%%%%%%%%%%%%%%%%%%%%%%%%%%%%%%%%%%%%%%%%%%%%%%%%%%%%%%%%%%%%%%%%%%%%%%%%%%%%%%%%%%%%%
\section{Technical steps to obtain the estimate for the cross-product}\label{tecgnicalstepcrossproduct}

In this section, we prove the technical results used in the development of the discrete Carleman estimate. We consider, in the following Lemmas, $\tau h(\delta T^{2})^{-1}\leq 1$ to ensure that the results from Section \ref{sec:weight:function} hold.  The development of the estimates presented in Lemmas \ref{lem:dz}, \ref{lem:AdditionalsTerms}, and \ref{lem:dt} will be divided into three sections. In the first section, we will define the stochastic terms, which involve the differential \( dz \) and include \( I_{11}, I_{21}, \) and \( I_{31} \). The second section will introduce the corrector terms, which are added to help us overcome the difficulties encountered in both the stochastic and deterministic terms, respectively. Finally, the third section will address the deterministic terms, which involve the differential \( dt \), namely \( I_{12}, I_{22}, \) and \( I_{32} \).

\subsection{Terms involving the differential dz.}\label{subsec:dz}
\subsubsection{Estimate of \texorpdfstring{$I_{11}$.}{}}
We set $\beta_{11}=rD_i^2\rho$. Then, using \eqref{eq:averengeanddifference} in $rA_i^2(\rho)$ and noticing that $r\rho=1$, it follows that
\begin{equation*}
\begin{split}
I_{11}=2\E \int_{Q} C_1(z)B_1(z)=& 2\sum_{i=1}^{n}\left(\E \int_{Q} D_i(\gamma_iD_iz)\,dz+\frac{h^2}{4}\E\int_{Q}\beta_{11}\, D_i(\gamma_iD_iz)\,dz\right)\\
\triangleq &\sum_{i=1}^{n}\left(I_{11}^{(a)}+I_{11}^{(b)}\right).
\end{split}
\end{equation*}
Let us first consider $I_{11}^{(a)}$. We get by integration by parts with respect to the differential operator and noting that $z=0$ on $\partial_{i} Q$, for each $i=1,...,n$ that
\begin{equation*}
\begin{split}
    I_{11}^{(a)}=& -2\E \int_{Q^{\ast}_{i}} \gamma_{i}D_izD_i(dz).
\end{split}
\end{equation*}
Moreover, by virtue of the Itô formula, we infer from the integral above that
%{we have that
%\begin{equation}\label{Eq:itoD_iz}
%    \left.2\int_0^T \gamma_i D_iz\,dD_iz=\gamma_i |D_iz|\right|_0^T-\int_0^T \gamma_i\, d[D_iz, D_iz].
%\end{equation}
%Given that $d[D_iz,D_iz]=(D_i(dz))^2$ and \eqref{eq:ito:sum},  
\begin{equation*}
    I_{11}^{(a)}=-\E\left. \int_{\mathcal{M}_i^*}\gamma_{i}|D_iz|^2\right|_0^T+\E\int_{Q^{\ast}_{i}}\gamma_{i}|D_i(dz)|^2.
\end{equation*} %escribir la frontera de la integracion por parte

In turn, according to the definition of $I_{11}^{(b)}$ and the property \eqref{eq:int:dif} for each $i=1,...,n$, we see that
\begin{equation*}
\begin{split}
    I_{11}^{(b)}=&-\frac{h^2}{2}\E \int_{Q_i^{*}} \gamma_{i}D_iz\,D_i(\beta_{11} \; dz),
\end{split}
\end{equation*}
we have used that $z=0$ on $\partial_i Q$. By \eqref{eq:difference:product} in the integral above, yields
\begin{equation}\label{eq:I11b}
    I_{11}^{(b)}=-\frac{h^2}{2}[\E\int_{Q_i^{\ast}} \gamma_i D_i(\beta_{11})\,D_iz\,A_i(dz)+\E\int_{Q_i^{\ast}} \gamma_i A_i(\beta_{11})\,D_iz\,D_i(dz)]\triangleq I_{11}^{(b_1)}+I_{11}^{(b_2)}.
\end{equation}
Applying the Itô's formula in $I_{11}^{(b_2)}$ we obtain 
\begin{equation}\label{firstintegral}
\begin{split}
  I_{11}^{(b_2)}=-\frac{h^2}{4} \left.\E\int_{Q^{\ast}_i}\gamma_i A_i(\beta_{11})|D_iz|^2\right|_0^T&+\frac{h^2}{4}\E\int_{Q_{i}^{\ast}}\gamma_i A_i(\beta_{11})|D_i(dz)|^2\\
  &+\frac{h^2}{4}\E\int_{Q_{i}^{\ast}}\partial_t(\gamma_i A_i(\beta_{11}))|D_iz|^2dt.
  \end{split}
\end{equation}
Now, thanks to Ito's formula and \eqref{eq:derivadacuadrado}, we can assert that $d(D_i(z^2))=2A_i(dz)D_i(z)+2A_izD_i(dz)+2A_i(dz)D_i(dz).$ 
This allows us to write $I_{11}^{(b_1)}$ as
\begin{equation}\label{eq:second:int1}
\begin{split}
 I_{11}^{(b_1)}=&-\frac{h^2}{4}\E\int_{Q_i^{\ast}} \gamma_i D_i(\beta_{11})d(D_i(z^2))+\frac{h^2}{2}\E\int_{Q_i^{\ast}} \gamma_i D_i(\beta_{11})A_i(dz)D_i(dz)\\
 &+\frac{h^2}{2}\E\int_{Q_i^{\ast}} \gamma_i D_i(\beta_{11})A_izD_i(dz)\\
 \triangleq& I_{11}^{(b_{11})}+I_{11}^{(b_{12})}+I_{11}^{(b_{13})}.
  \end{split}
\end{equation}
By \eqref{eq:int:dif}, Itô's formula and noting that $z=0$ on $\partial_i Q$, from $I_{11}^{b_{11}}$ we have 
\begin{equation}\label{eq:firstiib2}
\begin{split}
    I_{11}^{(b_{11})}=&\frac{h^2}{4}\E \int_{Q}D_i(\gamma_iD_i(\beta_{11}))d|z|^2-\frac{h^2}{4}\int_{\partial Q}d|z|^2t_{r}^{i}(\gamma_i D_i(\beta_{11}))\nu_i\\
    =& \left.\frac{h^2}{4}\E \int_{\mathcal{M}}D_i(\gamma_iD_i(\beta_{11}))|z|^2\right|_0^T-\frac{h^2}{4}\E \int_{Q}\partial_t(D_i(\gamma_iD_i(\beta_{11})))|z|^2dt.
    \end{split}
\end{equation}
Now we estimate the second integral in \eqref{eq:second:int1}. In view of \eqref{eq:derivadacuadrado} we note that
\begin{equation*}
   I_{11}^{(b_{12})}=\frac{h^2}{4}\E\int_{Q_i^{*}} \gamma_{i}D_i(\beta_{11})D_i(|dz|^2). 
\end{equation*}
Applying \eqref{eq:int:dif}, yields 
\begin{equation}\label{eq:second:int2}
\begin{split}
    I_{11}^{(b_{12})}=&-\frac{h^2}{4}\E\int_{Q}D_i[\gamma_iD_i(\beta_{11})]|dz|^2,
\end{split}
\end{equation}
where we utilized that $z=0$ on $\partial_i Q$. Analogously, using \eqref{eq:int:dif} on $I_{11}^{b_{13}}$ we get
\begin{equation}\label{eq:second:int3}
    \begin{split}
        I_{11}^{(b_{13})}=&-\frac{h^2}{2}\E\int_{Q} D_i(\gamma_i D_i(\beta_{11})A_iz)dz.
    \end{split}
\end{equation}
Replacing \eqref{eq:firstiib2}-\eqref{eq:second:int3} in  \eqref{eq:second:int1}, from $I_{11}^{b_1}$ we arrive at 
\begin{equation}\label{eq:secondintegral}
    \begin{split}
 I_{11}^{(b_1)}=&\left.\frac{h^2}{4}\E \int_{\mathcal{M}}D_i(\gamma_iD_i(\beta_{11}))|z|^2\right|_0^T-\frac{h^2}{4}\E \int_{Q}\partial_t(D_i(\gamma_iD_i(\beta_{11})))|z|^2dt\\
 &-\frac{h^2}{4}\E\int_{Q}D_i[\gamma_iD_i(\beta_{11})]|dz|^2-\frac{h^2}{2}\E\int_{Q} D_i(\gamma_i D_i(\beta_{11})A_iz)\,dz.
  \end{split}
\end{equation}
Therefore, combining \eqref{firstintegral} and \eqref{eq:secondintegral}, we can rewrite $I_{11}^{b}$ as
\begin{equation}
    \begin{split}
        &I_{11}^{(b)}=-\frac{h^2}{4}\left[ \left.\E\int_{\mathcal{M}^{\ast}_i}\gamma_i A_i(\beta_{11})|D_iz|^2\right|_0^T-\left.\E \int_{\mathcal{M}}D_i(\gamma_iD_i(\beta_{11}))|z|^2\right|_0^T\right.\\
        &-\E\int_{Q_{i}^{\ast}}\gamma_i A_i(\beta_{11})|D_i(dz)|^2+2\E\int_{Q} D_i(\gamma_i D_i(\beta_{11})A_iz)\,dz+\E\int_{Q}D_i[\gamma_iD_i(\beta_{11})]|dz|^2\\
 &\left.+\E \int_{Q}\partial_t(D_i(\gamma_iD_i(\beta_{11})))|z|^2dt-\E\int_{Q_{i}^{\ast}}\partial_t(\gamma_i A_i(\beta_{11}))|D_iz|^2dt\right].
    \end{split}
\end{equation}
Thus, using the definition of $C_4z$ and $B_1z$ on the last integral above, we can rewrite  $I_{11}$ as
\begin{equation}\label{eq:int:I11}
    \begin{split}
  &I_{11}=\sum_{i=1}^{n}\left(-\E\left. \int_{\mathcal{M}_i^*}\gamma_{i}|D_iz|^2\right|_0^T+\E\int_{Q^{\ast}_{i}}\gamma_{i}|D_i(dz)|^2-\frac{h^2}{4}\left[ \left.\E\int_{\mathcal{M}^{\ast}_i}\gamma_i A_i(\beta_{11})|D_iz|^2\right|_0^T \right.\right. \\
  &-\E\int_{Q_{i}^{\ast}}\gamma_i A_i(\beta_{11})|D_i(dz)|^2-\E\int_{Q_{i}^{\ast}}\partial_t(\gamma_i A_i(\beta_{11}))|D_iz|^2dt-\left.\E \int_{\mathcal{M}}D_i(\gamma_iD_i(\beta_{11}))|z|^2\right|_0^T\\
  &\left.\left.+\E \int_{Q}\partial_t(D_i(\gamma_iD_i(\beta_{11})))|z|^2dt+\E\int_{Q}D_i[\gamma_iD_i(\beta_{11})]|dz|^2\right]\right)-2\E\int_{Q} C_4z B_1z, 
\end{split}
\end{equation}
thank to $I_{11}^{(a)}$ and $I_{11}^{(b)}$. Now, we have to estimate the terms with $\beta_{11}=rD_i^2\rho$ on the previous expression, using Theorem \ref{theo:weight:estimates} and Lemma \ref{lem:derivative:wrt:x} we obtain
\begin{equation}
    \begin{aligned}
     &A_i(\beta_{11})=s^2\varphi^2\lambda^2(\partial_{i}\psi)^{2}+2s\mathcal{O}_{\lambda}(1)+s^2\mathcal{O}_{\lambda}((sh)^2)=s^2\mathcal{O}_{\lambda}(1),  \\
     &D_i(\beta_{11})=2s^2\varphi^2\lambda^3(\partial_i\psi)^3+s^2\varphi^2\lambda^2\mathcal{O}(1)+2s\mathcal{O}_{\lambda}(1)+s^2\mathcal{O}_{\lambda}((sh)^2)=s^2\mathcal{O}_{\lambda}(1),\\
     &\partial_t(\gamma_iA_i(\beta_{11}))=T\theta s^2\mathcal{O}_{\lambda}(1),
    \end{aligned}
\end{equation}
where in the last estimate we have used that $\gamma_i$ is time independent. Moreover, noting that $D_i(\gamma_i)=\mathcal{O}(1)$, $A_i(\gamma_i)=\mathcal{O}(1)$ 
 and $D_i(\gamma_iD_i(\beta_{11})=D_i(\gamma_i)A_iD_i(\beta_{11})+A_i(\gamma_i)D^2_i(\beta_{11})$, then thanks to Theorem \ref{theo:weight:estimates}
\begin{equation}\label{funtiongammaimport}
\begin{aligned}
&D_i(\gamma_iD_i(\beta_{11}))=s^2\mathcal{O}_{\lambda}(1),\\
&\partial_t(D_i(\gamma_iD_i(\beta_{11})))=T\theta s^2\mathcal{O}_\lambda(1).
\end{aligned}
\end{equation}
Combining these estimates, \eqref{eq:int:I11}  and recalling the definition of $I_{41}$, we conclude that
\begin{equation*}
    I_{11}+I_{41}=\sum_{i=1}^{n} \E\int_{Q_i^{\ast}}\gamma_i|D_i(dz)|^2+X_{11}-Y_{11},
\end{equation*}
where 
\begin{align*}
    X_{11}\triangleq &\sum_{i=1}^{n}\left(\E\int_{Q_{i}^{\ast}}\mathcal{O}_{\lambda}((sh)^2)|D_i(dz)|^2+\E\int_{Q_{i}^{\ast}}T\theta(t)\mathcal{O}_{\lambda}((sh)^2)\,|D_iz|^2dt\right)\\
    &-\E\int_{Q}T\theta(t) s^2\mathcal{O}_{\lambda}(1)|z|^2dt-\E\int_{Q}\mathcal{O}_{\lambda}((sh)^2)\,|dz|^2
\end{align*}
and
\begin{align*}
    Y_{11}\triangleq \sum_{i=1}^{n}\E\left. \int_{\mathcal{M}_i^*}[\mathcal{O}(1)+\mathcal{O}_{\lambda}((sh)^2)]\,|D_iz|^2\right|_0^T+\E\left.\int_{\mathcal{M}}\mathcal{O}_{\lambda}((sh)^2)\,|z|^2\right|_0^T.
\end{align*}
%%%%%%%%%%%%%%%%%%%%%%%%%%%%%%%%%%%%%%%%%%%%%%%%%%
%%%%%%%%%%%%%%%%%%%%%%%%%%%%%%%%%%%%%%%%
\subsubsection{Estimate of \texorpdfstring{$I_{21}$}{}}
We set $\beta_{21}=\gamma_irD_i^2\rho$. From the definition of $I_{21}$ and using \eqref{eq:averengeanddifference} in $A_i^2(z)$, it follows that
\begin{equation*}
\begin{split}
I_{21}=2\E \int_{Q} C_2(z)\,B_1(z)=& 2\sum_{i=1}^{n}\left(\E \int_{Q} \beta_{21}\; zdz+\frac{h^2}{4}\E\int_{Q}\beta_{21}\, D^2_iz\,dz\right)\\\triangleq&\sum_{i=1}^{n}\left(I_{21}^{(a)}+I_{21}^{(b)}\right).
\end{split}
\end{equation*}
By the Itô's formula from $I_{21}^{(a)}$ for each $i\in\llbracket 1,n \rrbracket$, we have that
\begin{equation*}
    \begin{split}
        I_{21}^{(a)}=\left.\E\int_{\mathcal{M}}\beta_{21}  |z|^2\right|_0^T-\E\int_{Q}\partial_t(\beta_{21}  )|z|^2dt- \E\int_{Q}\beta_{21} |dz|^2.
    \end{split}
\end{equation*}
On the other hand, applying \eqref{eq:int:dif} in $I_{21}^{(b)}$, we obtain for each i $=1,...,n$ that 
\begin{equation}
    I_{21}^{(b)}=-\frac{h^2}{2}\E \int_{Q_{i}^{\ast}} D_iz\,D_i(\beta_{21}\, dz)+\E\int_{\partial_{i} Q}dz\, \beta_{21}\, t_r^i(D_iz)\nu_{i}.
\end{equation}
Noting that $z=0$ on $\partial_i Q$ and using \eqref{eq:difference:product}, we see that
\begin{equation}
    I_{21}^{(b)}=-\frac{h^2}{2}\left[\E \int_{Q_i^{\ast}} D_i(\beta_{21})D_i z\; A_i(dz)+\E \int_{Q_i^{\ast}} A_i(\beta_{21})D_i z\; D_i(dz)\right]
\end{equation}
Analogously to $I_{11}^{(b)}$, from $I_{21}^{b}$ we arrived at
\begin{equation}
    \begin{split}
        I_{21}^{(b)}=&-\frac{h^2}{4}\left[ \left.\E\int_{\mathcal{M}^{\ast}_i}A_i(\beta_{21})|D_iz|^2\right|_0^T-\E\int_{Q_{i}^{\ast}} A_i(\beta_{21})|D_i(dz)|^2\right.\\
        &-\E\int_{Q_{i}^{\ast}}\partial_t(A_i(\beta_{21}))|D_iz|^2dt-\left.\E \int_{\mathcal{M}}D^2_i(\beta_{21}))|z|^2\right|_0^T+\E \int_{Q}\partial_t(D^2_i(\beta_{21}))|z|^2dt\\
 &\left.+\E\int_{Q}D^2_i(\beta_{21})]|dz|^2+2\E\int_{Q} D_{i}(D_i(\beta_{21})A_iz)\,dz\right].
    \end{split}
\end{equation}
Therefore, using the definition of $C_4(z)$, $B_1(z)$ on the last integral above and thanks to $I_{21}^{(a)}$ and $I_{21}^{(a)}$, we can rewrite $I_{11}$ as
\begin{equation}
    \begin{split}
        &I_{21}=-2\E\int_{Q} C_5(z)\, B_1(z)+\sum_{i=1}^{n}\left(\left.\E\int_{\mathcal{M}}\beta_{21}  |z|^2\right|_0^T-\E\int_{Q}\partial_t(\beta_{21}   )|z|^2dt- \E\int_{Q}\beta_{21}  |dz|^2\right.\\
        &-\frac{h^2}{4}\left[ \left.\E\int_{\mathcal{M}^{\ast}_i}A_i(\beta_{21})|D_iz|^2\right|_0^T\right.\left.-\E\int_{Q_{i}^{\ast}} A_i(\beta_{21})|D_i(dz)|^2\right.-\E\int_{Q_{i}^{\ast}}\partial_t(A_i(\beta_{21}))|D_iz|^2dt\\
        &\left.-\left.\E \int_{\mathcal{M}}D^2_i(\beta_{21}))|z|^2\right|_0^T\left.+\E \int_{Q}\partial_t(D^2_i(\beta_{21}))|z|^2dt+\E\int_{Q}D^2_i(\beta_{21})|dz|^2\right]\right).
    \end{split}
\end{equation}
The result follows by using Theorem \ref{theo:weight:estimates} and Lemma \ref{lem:derivative:wrt:x} on the terms with $\beta_{21}=\gamma_irD_i^2\rho$. Indeed, 
\begin{equation}\label{eq:estimate:I21a}
    \begin{aligned}   \beta_{21}&=\gamma_is^2\lambda^2\varphi^2(\partial_i\psi)^2+2s\mathcal{O}_{\lambda}(1)+s^2\mathcal{O}_{\lambda}((sh)^2)=s^2\mathcal{O}_{\lambda}(1),\\
        \partial_{t}(\beta_{21})&= T\theta s^2\mathcal{O}_{\lambda}(1).
    \end{aligned}
\end{equation}
Now, applying \eqref{eq:average:product}, we obtain
\begin{equation*}
A_i(\beta_{21})=A_i(\gamma_i)A_i(rD_i^2\rho)+\frac{h^2}{4}D_i(\gamma_i)D_i(rD_i^2\rho)\\
\end{equation*}
and by \eqref{eq:difference:product} we have
\begin{equation*}D_i^2(\beta_{21})=A_i^2(\gamma_i)D_i^2(rD_i^2\rho)+2A_iD_i(\gamma_i)D_iA_i(rD_i^2\rho)+D_i^2(\gamma_i)A_i(rD_i^2\rho).
\end{equation*}
Consider that $A^\alpha_i(\gamma_i)=D_i^{\alpha}(\gamma_i)=A_iD_i(\gamma_i)=\mathcal{O}(1)$ for $\alpha=1,2$, and repeated application of Theorem \ref{theo:weight:estimates} and Lemma \ref{lem:derivative:wrt:x}, we obtain the following estimates
\begin{equation}\label{eq:estimate:I21b}
    \begin{aligned}
&A_i(\beta_{21})=D_i^2(\beta_{21})=s^2\mathcal{O}_{\lambda}(1),\\    &\partial_t(A_i(\beta_{21}))=\partial_t(D_i^2(\beta_{21}))=Ts^2\theta\mathcal{O}_{\lambda}(1).
    \end{aligned}
\end{equation}
Combining the definition of $I_{51}$, \eqref{eq:estimate:I21a} and \eqref{eq:estimate:I21b}, we deduce that
\begin{equation*}
    I_{21}+I_{51}=-X_{21}-Y_{21},
\end{equation*}
where
\begin{equation*}
    \begin{split}
        X_{21}\triangleq&\; \E\int_{Q} T\theta [s^2\mathcal{O}_{\lambda}(1)-\mathcal{O}_{\lambda}((sh)^2)]|z|^2dt+\E\int_{Q}(s^2\mathcal{O}_{\lambda}(1)-\mathcal{O}_{\lambda}((sh)^2))|dz|^2\\
        &+\sum_{i=1}^{n}\left(\E\int_{Q_{i}^{\ast}}\mathcal{O}_{\lambda}((sh)^2)|D_i(dz)|^2+\E\int_{Q_i^{\ast}}T\theta\mathcal{O}_{\lambda}((sh)^2)|D_iz|^2dt\right)
    \end{split}
\end{equation*}
and
\begin{equation*}
\begin{split}
    Y_{21}\triangleq&\;\left.\E\int_{\mathcal{M}}[\mathcal{O}_\lambda((sh)^2)-s^2\mathcal{O}_{\lambda}(1)]|z|\right|_0^T+\left.\sum_{i=1}^{n}\E\int_{\mathcal{M}_{i}^{\ast}}\mathcal{O}_{\lambda}((sh)^2)|D_iz|^2\right|_0^T.
\end{split}
\end{equation*}
%%%%%%%%%%%%%%%%%%%%%%%%%%%%%%%%%%%%%%%%%%%%%%%%%%%%%%%%%%%%%%%%%%%%%%%%%%%%%%%%%%%%%%%%%%%%%%%%%%%%%%%%%%%%%%%%%%%%%%%%%%%%%%
\subsubsection{Estimate of \texorpdfstring{$I_{31}$}{}}
Similarly to $I^{(a)}_{21}$, we obtain that
\begin{equation*}
I_{31}=\left.\E\int_{\mathcal{M}}r\partial_{t}(\rho)\,  |z|^2\right|_0^T-\E\int_{Q}\partial_t(r\partial_{t}(\rho))\,|z|^2dt- \E\int_{Q}r\partial_{t}(\rho)\, |dz|^2.
\end{equation*}
Consider $|\theta_{t}|\leq CT\theta^2$ and $|\theta_{tt}|\leq CT\theta^3$, thus
$I_{31}\geq -X_{31}-Y_{31}$,
where
\begin{equation*}
X_{31}\triangleq\,\E\int_{Q}T\theta^3\mathcal{O}(1)|z|^2dt+\E\int_{Q}T\theta^2\mathcal{O}(1)|dz|^2 \text{ and } Y_{31}\triangleq\,\left.\E\int_{\mathcal{M}}T\theta^2\mathcal{O}(1)|z|^2\right|_0^T. 
\end{equation*}
%%%%%%%%%%%%%%%%%%%%%%%%%%%%%%%%%%%%%%%%%%%%%%%%%%%%%%%%%%%%%%%%%%%%%%%%%%%%%%%%%%%%%%%%%%%%%%%%%%%%%%%%%%%%%%%%%%%%%%%%%%%%%%%%%%%%%%%%%%%%%%%%%%%%%%%%%%%%%%%%%%%%%%%%%%%%%%%%%%%%%%%%%%%%%%%%%%%%%%%%%%%%%%%%

\subsection{Products with supplementary terms.}

\subsubsection{Estimate of \texorpdfstring{$I_{42}$}{}}
%%%%%%%%%%%%%%%%%%%%%%%%%%%%%%%%%%%%%
Denoting $\beta_{42}^{ij}\triangleq\,\gamma_jrD_i(\gamma_i D_i(rD_i^2\rho)\,D_jA_j\rho$, $\alpha_{42}^{i j}\triangleq\,\gamma_j r A_i(\gamma_i D_i(r D_i^2\rho)\,D_j A_j\rho$ and using \eqref{eq:difference:product} on $C_4(z)$, we have that 
\begin{equation*}
\begin{split}
    2\E \int_Q C_4(z)\,B_2(z)\,dt=&h^2\sum_{i,j=1}^{n}\E\int_{Q}(\beta_{42}^{ij}A_i^2z\,D_jA_jz+\alpha_{42}^{ij}\,D_iA_iz\,D_jA_jz)\,dt\\
    &\triangleq\,\sum_{i,j=1}^{n} I_{42}^{ij}.
\end{split}
\end{equation*}

Let us first, we will estimate $\beta_{42}^{ij}$ and $\alpha_{42}^{ij}$ for each $i,j=1,...,n$. From Theorem \ref{eq1:theo:weight:estimates} and Lemma \ref{lem:derivative:wrt:x} for each $i,j=1,...,n$, we obtain that
\begin{equation}\label{eq:estimatealphabeta42}
    \begin{split}
        \beta_{42}^{ij}=\alpha_{42}^{ij}=s^3\mathcal{O}_{\lambda}(1).
    \end{split}
\end{equation}
Now, we will focus on the case $i=j$.  Applying Young's inequality and using \eqref{eq:estimatealphabeta42} we note that
\begin{equation*}
    |I_{42}^{ii}|\leq \E\int_{Q}s|\mathcal{O}_{\lambda}((sh)^2)|\left(|A_i^2z|^2+|D_iA_iz|^2\right)\;dt
\end{equation*}
By \eqref{eq:promedioinequality} and \eqref{eq:int:ave}, we obtain
\begin{equation*}
    \begin{split}
        \E\int_{Q}s|\mathcal{O}_\lambda((sh)^2|&\, |A_i^2z|^2\,dt \leq \E\int_{Q}s|\mathcal{O}_\lambda((sh)^2|\, A_i(|A_iz|^2)\,dt\\
        \leq & \E\int_{Q_{i}^{\ast}}s|\mathcal{O}_\lambda((sh)^2|\, |A_iz|^2\,dt-\frac{h}{2}\E\int_{\partial_i Q} s|\mathcal{O}_{\lambda}((sh)^2)|t_r^i(|A_iz|^2)\\
        \leq & \E\int_{Q_i^{\ast}}s|\mathcal{O}_\lambda((sh)^2|\, A_i(|z|^2)\,dt,
\end{split}
\end{equation*}
where we used that $t_r^i(|A_iz|^2)\geq 0$. Using again \eqref{eq:promedioinequality}, \eqref{eq:int:ave} and noting that $z=0$ on $\partial_i Q$ for each $i\in\llbracket 1,n \rrbracket$, we see that
\begin{equation}\label{eq:estimates:I_{42(a)}}
\begin{split}
   \E\int_{Q}s|\mathcal{O}_\lambda((sh)^2|&\, |A_i^2z|^2\,dt     \leq  \E\int_{Q_i^{\ast}}s|\mathcal{O}_\lambda((sh)^2|\, A_i(|z|^2)\,dt\\
   \leq & \E\int_{Q}s|\mathcal{O}_\lambda((sh)^2|\, |z|^2\,dt+\frac{h}{2}\E\int_{\partial_i Q} s|\mathcal{O}_{\lambda}((sh)^2)|\,|z|^2\,dt\\
        \leq &\E\int_{Q}s|\mathcal{O}_\lambda((sh)^2|\, |z|^2\,dt.
    \end{split}
\end{equation}
Similarly, by \eqref{eq:promedioinequality}, \eqref{eq:int:ave} and given that $t_r^i(|D_iz|^2)\geq 0$ on $\partial_i Q$ for each $i=1,...,n$, yields
\begin{equation}\label{eq:estimates:I_{42(b)}}
    \begin{split}
        \E\int_{Q}s|\mathcal{O}_\lambda((sh)^2|&\, |A_iD_iz|^2\,dt \leq \E\int_{Q}s|\mathcal{O}_\lambda((sh)^2|\, A_i(|D_iz|^2)\,dt\\
        \leq & \E\int_{Q_{i}^{\ast}}s|\mathcal{O}_\lambda((sh)^2|\, |D_iz|^2\,dt-\frac{h}{2}\E\int_{\partial_i Q} s|\mathcal{O}_{\lambda}((sh)^2)|t_r^i(|D_iz|^2)\,dt\\
        \leq & \E\int_{Q_i^{\ast}}s|\mathcal{O}_\lambda((sh)^2|\, |D_iz|^2\,dt.\\
    \end{split}
\end{equation}
Thanks to \eqref{eq:estimates:I_{42(a)}} and \eqref{eq:estimates:I_{42(b)}}, we  can assert that
\begin{equation}\label{eq:estimates:I_{42}ii}
    I_{42}^{ii}\geq -\E\int_{Q}s|\mathcal{O}_\lambda((sh)^2|\, |z|^2\,dt-\E\int_{Q_i^{\ast}}s|\mathcal{O}_\lambda((sh)^2|\, |D_iz|^2\,dt \triangleq -W^{ii}_{42}.
\end{equation}

We now turn to the case $i\not=j$. Consider $$I_{42}^{ij}=h^2\E\int_{Q}(\beta_{42}^{ij}A_i^2z\,D_jA_jz+\alpha_{42}^{ij}\,D_iA_iz\,D_jA_jz)\,dt\triangleq\,I_{42}^{ij(a)}+I_{42}^{ij(b)},$$
using Young's inequality on $I_{42}^{ij(a)}$ and the estimates for $\beta_{42}^{ij}$ and $\alpha_{42}^{ij}$ from \eqref{eq:estimatealphabeta42}, gives
\begin{equation*}
    |I_{42}^{ij(a)}|\leq\,\E\int_{Q}s|\mathcal{O}_{\lambda}((hs)^2)|\,\left(|A_i^2z|^2+|D_jA_jz|^2\right)\,dt
\end{equation*}
Hence, combining \eqref{eq:estimates:I_{42(a)}} with \eqref{eq:estimates:I_{42(b)}}, we deduce that
\begin{equation}
    |I_{42}^{ij(a)}|\leq\,\E\int_{Q}s|\mathcal{O}_\lambda((sh)^2|\, |z|^2\,dt+\E\int_{Q_j^{\ast}}s|\mathcal{O}_\lambda((sh)^2|\, |D_jz|^2\,dt.
\end{equation}
Repeated application of Young's inequality and \eqref{eq:estimates:I_{42(b)}} on $I_{42}^{ij(b)}$ enables us to write
\begin{equation}\label{eq:estimates:I_{42}ij}
|I_{42}^{ij(b)}|\leq\,\E\int_{Q_i^{\ast}}s|\mathcal{O}_\lambda((sh)^2|\, |D_iz|^2\,dt+\E\int_{Q_j^{\ast}}s|\mathcal{O}_\lambda((sh)^2|\, |D_jz|^2\,dt.
\end{equation}
From \eqref{eq:estimates:I_{42}ii}-\eqref{eq:estimates:I_{42}ij}, it may be concluded that
\begin{equation}\label{eq:estimatefinalI_42}
    I_{42}\geq\,-\sum_{i=1}^{n}\E\int_{Q_i^{\ast}}s|\mathcal{O}_\lambda((sh)^2|\, |D_iz|^2\,dt-\E\int_{Q}s|\mathcal{O}_\lambda((sh)^2|\, |z|^2\,dt.
\end{equation}
\subsubsection{Estimate of \texorpdfstring{$I_{52}$}{}}
%%%%%%%%%%%%%%%%%%%%%%%%%%%%%%%%%%%%%
Denoting $\beta_{52}^{ij}\triangleq\,\gamma_jrD^2_i(\gamma_i rD_i^2\rho)\,D_jA_j\rho$, $\alpha_{52}^{ij}\triangleq\,\gamma_jrA_i( D_i(\gamma_irD_i^2\rho)\,D_jA_j\rho$ and using \eqref{eq:difference:product} on $C_{5}(z)$, we obtain
\begin{equation*}
    2\E \int_Q C_5(z)\,B_2(z)\,dt=h^2\sum_{i,j=1}^{n}(\beta_{52}^{ij}A_i^2z\,D_jA_jz+\alpha_{52}^{ij}\,D_iA_iz\,D_jA_jz)\,dt\triangleq\,\sum_{i,j=1}^{n}I_{52}^{ij}.
\end{equation*}

Since $\beta_{52}^{ij}=\alpha_{52}^{ij}=s^3\mathcal{O}_{\lambda}(1)$ for each $i,j=1,...,n$, thank to the Theorem \ref{theo:weight:estimates} and Lemma \ref{lem:derivative:wrt:x}.  This allow us to estimate $I_{52}$ as $I_{42}$. Therefore, 
\begin{equation}\label{eq:estimatefinalI52}
\begin{split}
    I_{52}\geq\,&-\sum_{i=1}^{n}\E\int_{Q_i^{\ast}}s|\mathcal{O}_\lambda((sh)^2|\,|D_iz|^2\,dt-\E\int_{Q}s|\mathcal{O}_\lambda((sh)^2|\, |z|^2\,dt.
\end{split}
\end{equation}
%%%%%%%%%%%%%%%%%%%%%%%%%%%%%%%%%%%%%%%%%%%%%%%%%%%%%%%%%%
%%%%%%%%
\subsubsection{Estimate of \texorpdfstring{$I_{33}$}{}}
%%%%%%%%%%%%%%%%%%%%%%%%%%%%%%%%%%%%%%%%%%%%%%%%%%%%%%%%%%%%%%%%%%%%%%%%%%%%%%%%%%%%%%%%%%%%%%%%%%%%%%%%%%%%%%%
Denoting $\beta_{22}\triangleq\,-2sr\partial_t(\varphi)\Delta_{\gamma}(\varphi)$  we have $\displaystyle
        2\E \int_{Q} C_{3}zB_{3}zdt=2\E\int_{Q}\beta_{22}\,|z|^2dt\triangleq\,I_{33}$.
Consider $|\theta_t|\leq CT\theta^2$ and observing that
$\Delta_{\gamma}(\varphi)=\lambda^2\varphi|\nabla_{\gamma}\psi|^2+\lambda \varphi \mathcal{O}(1)$, thus
\begin{equation}\label{eq:estimatefinalI33}
    I_{33}\geq -\E\int_{Q}\left[Cs\lambda^2T\theta^2\varphi|\nabla_{\gamma}\psi|^2+s\lambda T\theta^2\varphi \mathcal{O}(1)\right]|z|^2\,dt.
\end{equation}
%%%%%%%
\subsubsection{Estimate of \texorpdfstring{$I_{43}$}{}}
%%%%%%

By setting $\beta_{43}^{i}\triangleq\,sD_i(\gamma_i D_i(rD_i^2\rho)\,\Delta_{\gamma}(\varphi)$, $\alpha^{i}_{43}\triangleq\,s A_i(\gamma_i D_i(r D_i^2\rho)\,\Delta_{\gamma}(\varphi)$ and using \eqref{eq:difference:product}, we have that 
\begin{equation*}
    2\E \int_Q C_4(z)\,B_3(z)\,dt=-h^2\sum_{i=1}^{n}\E\int_{Q}(\beta_{43}^{i}A_i^2(z)\,z+\alpha_{43}^{i}\,D_i(A_iz)\,z)\,dt\triangleq\,\sum_{i=1}^{n}I_{43}^{i}.
\end{equation*}

By Theorem \ref{theo:weight:estimates} and Lemma \eqref{lem:derivative:wrt:x}, we have $\beta_{43}^{i}=\alpha_{43}^{i}=s^3\mathcal{O}_{\lambda}(1)$ for each $i\in \llbracket 1,\rrbracket$. Then, using Young's inequality, \eqref{eq:estimates:I_{42(a)}} and \eqref{eq:estimates:I_{42(b)}}, we obtain
\begin{equation*}
\begin{split}
    |I_{43}^{i}|\leq & \,s\mathcal{O}_{\lambda}((sh)^2) \E\int_{Q}|A_i^2z|^2+|D_iA_iz|^2+|z|^2\,dt\\
    \leq& \E\int_{Q}s|\mathcal{O}_{\lambda}((sh)^2)|\,|z|^2\,dt+\E\int_{Q_{i}^{\ast}}s|\mathcal{O}_{\lambda}((sh)^2)|\,|D_iz|^2\,dt.
\end{split}    
\end{equation*}
Therefore,
\begin{equation}\label{eq:estimatefinlI43}
    I_{43}=\sum_{i=1}^{n}I_{43}^{i}\geq -\E\int_{Q}s|\mathcal{O}_{\lambda}((sh)^2)|\,|z|^2\,dt-\sum_{i=1}^{n}\E\int_{Q_{i}^{\ast}}s|\mathcal{O}_{\lambda}((sh)^2)|\,|D_iz|^2\,dt.
\end{equation}
%%%%%%%%%
\subsubsection{Estimate of \texorpdfstring{$I_{53}$}{}}
%%%%%%%%%

Setting $\beta_{52}^{i}\triangleq\,sD^2_i(\gamma_i rD_i^2\rho)\,\Delta_{\gamma}(\varphi)$, $\alpha_{52}^{i}\triangleq\,sA_i( D_i(\gamma_irD_i^2\rho)\,\Delta_{\gamma}(\varphi)$ and using \eqref{eq:difference:product}, we have to estimate
\begin{equation*}
    2\E \int_Q C_5(z)\,B_3(z)\,dt=-h^2\sum_{i=1}^{n}(\beta_{53}^{i}A_i^2(z)\,z+\alpha_{53}^{i}\,D_i(A_iz)\,z)\,dt\triangleq\,\sum_{i=1}^{n}I_{53}^{i}.
\end{equation*}
We can now proceed analogously to the estimate of $I_{43}$, getting
\begin{equation*}
    |I_{53}^{i}|\leq\,\E\int_{Q}s|\mathcal{O}_{\lambda}((sh)^2)|\,|z|^2\,dt+\E\int_{Q_{i}^{\ast}}s|\mathcal{O}_{\lambda}((sh)^2)|\,|D_iz|^2\,dt.
\end{equation*}
Therefore,
\begin{equation}\label{eq:estimatefinalI53}
    I_{53}=\sum_{i=1}^{n}I_{53}^{i}\geq -\E\int_{Q}s|\mathcal{O}_{\lambda}((sh)^2)|\,|z|^2\,dt-\sum_{i=1}^{n}\E\int_{Q_{i}^{\ast}}s|\mathcal{O}_{\lambda}((sh)^2)|\,|D_iz|^2\,dt.
\end{equation}

%%%%%%%%%%%%%%%%%%%%%%%%%%%%%%%%%%%%%%%%%%%%%%%%%%%%%%%%%%%%%%%%%%%%%%%%%%%%%%%%%%%%%%%%%%%%%%%%%%%%%%%%%%%%%%%%%%%%%%%%
\section{Technical steps to obtain the intermediate result and other result.}\label{sec:intermediateResults}
\subsection{A Discrete Sobolev-Type Inequality}\label{Sec:Sobolevinequality}

In this section, we establish two results from \cite[Section 9]{Brezis2011} in the discrete setting. The first is the lemma \ref{Lem:sobolevdiscrete}, and the second is Theorem \ref{teo:discreteSobolev}. We begin by defining a collection of discrete sets that will be useful throughout the formulation. Given $x \in \mathcal{M}$, we denote the set of points between $x$ and $x-(x_i-h)e_{i}$ as follows
$$
\mathcal{M}^{i}(x)\triangleq\left\{y\in \mathcal{M} : x - y = k h e_i\quad \text{with } k = 1, \dots, \frac{x_i}{h} - 1 \right\}.
$$
Observe that $\mathcal{M}^i(x) \subset \mathcal{M}$. Moreover, we can rewrite the set $\mathcal{M}$ in terms of the sets $\mathcal{M}^{i}$ since there exists $x^i \in \partial_i \mathcal{M}$ such that
\begin{equation}\label{eq:mesh}
\mathcal{M} = \prod_{i=1}^{n} \mathcal{M}^{i}(x^i).    
\end{equation}
Similarly, we also denote the dual set of $\mathcal{M}^{i}(x)$ and its boundary as $(\mathcal{M}^i(x))^{\ast}_{i}$ and $\partial_i \mathcal{M}^i(x)$, respectively. We notice that the dual sets $\mathcal{M}_{i}^{\ast}$ can be rewritten for $i=1,\ldots,n$, as follows.
\begin{equation}\label{eq:Mdual}
\mathcal{M}_i^{\ast} = \prod_{\substack{j = 1 \\ i \ne j}}^{n} \mathcal{M}^{j}(x^j) \times (\mathcal{M}^i(x^i))_{i}^{\ast}.
\end{equation}
For any $y \in \mathcal{M}$, we define the discrete integral along the $i$-th direction from $y-y_ie_i$ to $y$ as
$$
\int_{\mathcal{M}^{i}(y)} u(y_1,\dots,y_{i-1},s,y_{i+1},\dots,y_n) \triangleq h \sum_{x \in \mathcal{M}^{i}(y)} u(x).
$$
Furthermore, thanks to \eqref{eq:mesh} there exists $x^i \in \partial_i \mathcal{M}$, for $i=1,\ldots, n$, such that
\begin{equation}\label{eq:oneandmultidirection}
\int_{\mathcal{M}^1(x^1)} \cdots \int_{\mathcal{M}^n(x^n)} u = \int_{\mathcal{M}} u.
\end{equation}

The proof of Theorem \ref{teo:discreteSobolev} is based on a discrete version of Lemma 9.2 in \cite{Brezis2011}. Since the proof of this Lemma mimics step by step the strategy of the continuous setting, we omit it.
\begin{lemma}\label{Lem:sobolevdiscrete}
Let $n\geq2$, and let $f_i\in L^{n-1}_{h}\left(\mathcal{N}_i\right)$ for $i=1,...,n.$, where $\mathcal{N}_i=\displaystyle\prod_{\substack{j=1\\ j\ne i}}^{n}\mathcal{M}^j(x^j)$. For $x\in \mathcal{M}$ and $1\leq i \leq n$, define $
\tilde{x}_i=(x_1,x_2,...,x_{i-1},x_{i+1},...,x_n)
$, that is, $x_i$ is omitted from the list. Then, the function  
$
f(x)=f_1(\tilde{x}_1)f_2(\tilde{x}_2)\cdots f_{n}(\tilde{x}_n)$ belongs to $L_{h}^{1}(\mathcal{M}),$
and
\begin{equation}
\|f\|_{L_{h}^{1}(\mathcal{M})}\leq\prod_{i=1}^{n}\|f_i\|_{L_{h}^{n-1}\left(\mathcal{N}_i\right)}.
\end{equation}
\end{lemma}
The proof of Theorem \ref{teo:discreteSobolev} follows the same strategy as in the continuous case. However, we do not have a chain rule for the discrete difference operator $D_{i}(|u|^{m})$. To address this, we derive an identity for $D_{i}(|u||u|^{m-1})$ applying the Leibniz rule \eqref{eq:difference:product}. This approach introduces additional terms related to the average operator, making the adaptation nontrivial.
\subsubsection{Proof Theorem \ref{teo:discreteSobolev}.} 
We begin with the case $p=1$, where $u\in C(\overline{\mathcal{M}})$ satisfies $u(\vec{0})=0$. We have
\[
|u(x_1,x_2,...,x_i...,x_n)|=\left|\int_{(\mathcal{M}^i(x_i))_{i}^\ast}D_iu\right|\leq \int_{(\mathcal{M}^i)_i^\ast}|D_iu|,
\]
where $\mathcal{M}^i(x_i)$ is a subset of $\mathcal{M}^i$ such that if $x\in \mathcal{M}^i(x_i)$ then $x\cdot e_i< x_i$. Thus
\[
|u(x)|^n\leq \prod_{i=1}^{n}\int_{(\mathcal{M}^{i})^{\ast}_{i}}|D_{i}u|.
\]
We deduce from Lemma \ref{Lem:sobolevdiscrete}, with $\displaystyle f_{i}(\tilde{x}_{i})=\int_{(\mathcal{M}^{i})^{\ast}_{i}}|D_{i}u|$, 
\begin{align*}
\int_{\mathcal{M}}|u(x)|^{n/(n-1)}\leq& \prod_{i=1}^{n}\|f_i\|_{L_{h}^1(\mathcal{N}_i)}^{1/(n-1)}=\prod_{i=1}^{n}\left\|\int_{(\mathcal{M}^i)_i^\ast}|D_iu|\right\|_{L_{h}^1(\mathcal{N}_i)}^{1/(n-1)}.
\end{align*}
Thanks to \eqref{eq:Mdual}, we can notice that
\begin{align*}
\prod_{i=1}^{n}\left\|\int_{(\mathcal{M}^i)_i^\ast}|D_iu|\right\|_{L^1\left(\mathcal{N}_i\right)}^{1/(n-1)}=\prod_{i=1}^{n}\left(\int_{\mathcal{M}_i^{\ast}}|D_iu|\right)^{1/(n-1)}.
\end{align*}
Therefore,
\begin{equation*}
    \int_{\mathcal{M}}|u(x)|^{n/(n-1)}\leq \prod_{i=1}^{n}\left(\int_{\mathcal{M}_i^{\ast}}|D_iu|\right)^{1/(n-1)}.
\end{equation*}
As a consequence, we have the following.
\begin{align}\label{eq:sobolev}
    \left(\int_{\mathcal{M}}|u|^{n/(n-1)}\right)^{(n-1)/n}\leq \prod_{i=1}^{n}\left(\int_{\mathcal{M}_i^{\ast}}|D_iu|\right)^{1/n}.
\end{align}
The proof for $p=1$ follows after applying the inequality of arithmetic and geometric mean to the right-hand side above.

We now turn to the case $1<p<n$. Let $m\geq 1$; the idea is to apply \eqref{eq:sobolev} to $|u|^{m-1}u$. However, since we do not have a chain rule for the operator $D_{i}$, we will use the following identity.
\begin{align}
    D_i(|u|^{m-1}u)=&D_i(|u|^{m-1})A_i(u)+A_i(|u|^{m-1})D_i(u) \nonumber\\
    =&A_i(u)D_i(|u|)\left(\sum_{k=0}^{m-2}(A_i|u|)^{k}A_i(|u|^{m-2-k})\right)+A_i(|u|^{m-1})D_i(u). \label{eq:productsobolev}
\end{align}
The proof of \eqref{eq:productsobolev} is consequence of the expression
\begin{equation*}
    D_i(|u|^{m-1})=D_{i}(|u|)\sum_{k=0}^{m-2}(A_i|u|)^{k}A_i(|u|^{m-2-k}).
\end{equation*}
which follows by induction over $m\in\mathbb{N}$. Moreover, also by induction, we can prove  that for $m\in \mathds{N}$, holds the next inequality
\begin{align}\label{eq:averagesobolev}
    A_i(|u|^m)
    \leq& C(A_i|u|)^m.
    \end{align}
Hence, combining \eqref{eq:productsobolev} and \eqref{eq:averagesobolev}, and noting that $|D_i|u||\leq |D_iu|$ and $|A_iu|\leq A_i|u|$, we obtain
\begin{equation}\label{ine:Dum}
\begin{aligned}
   |D_i(|u|^{m-1}u)|\leq& |D_iu|\left(\sum_{k=1}^{m-2}(A_i|u|)^{k+1}A_i(|u|^{m-2-k})+A_i(|u|^{m-1})\right)\\
   \leq &|D_iu|(A_i|u|)^{m-1}\left(\sum_{k=0}^{m-2}2^{m-1-k}+2^{m-1}\right)\\
   \leq& C(A_i|u|)^{m-1}|D_iu|.
\end{aligned} 
\end{equation}
Thus, applying \eqref{eq:sobolev} to $|u|^{m-1}u$  we obtain 
\begin{align}\label{eq:sobolev1}
    \left(\int_{\mathcal{M}}|u|^{mn/(n-1)}\right)^{(n-1)/n}\leq \prod_{i=1}^{n}\left(\int_{\mathcal{M}_i^{\ast}}|D_i(|u|^{m-1}u)|\right)^{1/n}.
\end{align}
Therefore, using \eqref{ine:Dum} in \eqref{eq:sobolev1} it follows that
\begin{equation}\label{ine:almost:Sobolev}
\begin{aligned}
    \left(\int_{\mathcal{M}}|u|^{mn/(n-1)}\right)^{(n-1)/n}&\leq C \prod_{i=1}^{n}\left(\int_{\mathcal{M}_i^{\ast}}(A_i|u|)^{m-1}|D_iu|\right)^{1/n}\\
&\leq C \prod_{i=1}^{n}\mathcal{I}_{i}^{\frac{1}{np'}}\left(\int_{\mathcal{M}_i^{\ast}}|D_iu|^p\right)^{1/(np)},
\end{aligned}
\end{equation}
where $\frac{1}{p'}+\frac{1}{p}=1$ and $\displaystyle\mathcal{I}_{i}:=\int_{\mathcal{M}_i^{\ast}}(A_i|u|)^{p'(m-1)}$.  \\
The key step is to prove a uniform bound on $\mathcal{I}_{i}$ (independent of $i$) in terms of the integral of $u$ over $\mathcal{M}$. Using that $(a+b)^{c}\leq 2^{c-1}(a^c+b^c)$ for $a,b\geq 0$ and $c\geq 1$, we have
\begin{align*}
    \mathcal{I}_{i}\leq \int_{\mathcal{M}_i^{\ast}}2^{p'(m-1)-1}A_i(|u|^{p'(m-1)}),
\end{align*}
Now, performing an integration by parts with respect to the average operator and using that $u=0$ on $\partial\mathcal{M}$, we obtain
\begin{align*}
    \mathcal{I}_{i}\leq& \int_{\mathcal{M}}|u|^{p'(m-1)}A_i(2^{p'(m-1)-1})+\frac{h}{2}\int_{\partial_i\mathcal{M}}|u|^{p'(m-1)}t_r^i(2^{p'(m-1)-1})\\
    \leq&\int_{\mathcal{M}}2^{p'(m-1)-1} |u|^{p'(m-1)}.
\end{align*}
Then, using the above inequality in \eqref{ine:almost:Sobolev} allows us to conclude
\begin{align}\label{eq:sobolev2}
 \left(\int_{\mathcal{M}}|u|^{mn/(n-1)}\right)^{(n-1)/n}&\leq  C\left(\int_{\mathcal{M}}2^{p'(m-1)-1}|u|^{p'(m-1)}\right)^{\frac{1}{p'}}\prod_{i=1}^{n}\left(\int_{\mathcal{M}_i^{\ast}}|D_iu|^p\right)^{1/(np)}.
\end{align}
 Thus, choosing $m\in \mathds{N}$ such that $mn/(n-1) = p'(m- 1)$, that is $m:=(n-1)p^{\ast}/n\geq 1$ (since $1 < p < n$); we have
\begin{align*}
    \|u\|^m_{L_{h}^{p^{\ast}}(\mathcal{M})}\leq C\left(\int_{\mathcal{M}}|u|^{p^\ast}\right)^{\frac{m-1}{p^{\ast}}}\prod_{i=1}^{n}\left(\int_{\mathcal{M}_i^{\ast}}|D_iu|^p\right)^{1/(np)},
\end{align*}
implying
\begin{align*}
    \|u\|_{L_{h}^{p^{\ast}}(\mathcal{M})}\leq C\prod_{i=1}^{n}\left(\int_{\mathcal{M}_i^{\ast}}|D_iu|^p\right)^{1/(np)}.
\end{align*}
Finally, let us consider the case $p=n$. From \eqref{eq:sobolev2}, for all $m>1$ we have
$$
 \left(\int_{\mathcal{M}}|u|^{mn/(n-1)}\right)^{(n-1)/n} \leq C\left(\int_{\mathcal{M}}|u|^{\frac{n(m-1)}{n-1}}\right)^{\frac{(n-1)}{n}}\prod_{i=1}^{n}\left(\int_{\mathcal{M}_i^{\ast}}|D_iu|^n\right)^{1/(n^2)} ,
$$
and thanks to Young's inequality, we have for all $m>1$
$$
 \left(\int_{\mathcal{M}}|u|^{mn/(n-1)}\right)^{(n-1)/(nm)} \leq C\left(\left(\int_{\mathcal{M}}|u|^{\frac{n(m-1)}{n-1}}\right)^{\frac{n-1}{(m-1)n}}+\sum_{i=1}^n\left(\int_{\mathcal{M}_i^{\ast}}|D_iu|^n\right)^{1/n}\right).
$$
Choosing $m=n$ in the above inequality yields
$$
\|u\|_{L_h^{n^2 /(n-1)}(\mathcal{M})} \leq C\|u\|_{W_{h}^{1, n}(\mathcal{M})}.
$$
By the interpolation inequality, this implies
$$
\|u\|_{L_{h}^{p^{\ast}}(\mathcal{M})} \leq C\|u\|_{W_{h}^{1, n}(\mathcal{M})}
$$
for all $p^{\ast}\in [n,n^{2}/(n-1)]$. Iterating this argument with $m=n+1$, $m=n+2$, and so on, we conclude the desired result.
\subsection{Proof of Lemma \ref{lem:estimate:termDw}}
 We choose a function $\xi\in C_0^{\infty}(G_0;[0,1])$ such that $\xi\equiv 1$ in $G_1$ . By the Itô formula, then we see that
\begin{equation*}
    \begin{split}
        \E&\left.\int_{G_0\cap\mathcal{M}}s\varphi \xi^2 e^{2s\theta\varphi}|w|^2\right|_0^T\\
        &=\E\int_0^T\int_{G_0\cap\mathcal{M}}\xi^2\varphi\left[\partial_t(se^{2s\theta\varphi})|w|^2\,dt+2s e^{2s\theta\varphi}\,wdw+s e^{2\theta\varphi}|dw|^2\right].
    \end{split}
\end{equation*}
Recalling that $dw+\sum_{i=1}^{n}D_i(\gamma_i D_iw)dt=fdt+gdB(t)$ and noting that
\begin{align*}
&|dw|^2=dw\; dw=g^2\,(dB(t))^2+2\left[g\sum_{i=1}^{n}D_i(\gamma_iD_iw)+fg\right]dt\,dB(t)\\
&+\left[\left(\sum_{i=1}^{n}D_i(\gamma_iD_iw)\right)\left(\sum_{j=1,...,n}D_{j}(\gamma_jD_jw)\right)+f\sum_{i=1,...,n }D_i(\gamma_iD_i w)+f^2\right]\,(dt)^2\\
&=g^2\,dt,
\end{align*}
where we used the fact that $(dt)^2=dt\,dB(t)=0$ and $(dB(t))^2=dt$, we can rewrite the integral above as
\begin{equation}\label{eq1:lemma4.7}
    \begin{split}
        \E&\left.\int_{G_0\cap\mathcal{M}}s\varphi \xi^2 e^{2s\theta\varphi}|w|^2\right|_0^T=-2\sum_{i=1}^{n}\E\int_0^T\int_{G_0\cap\mathcal{M}}s\xi^2 \varphi e^{2\theta\varphi}wD_i(\gamma_iD_iw)\,dt\\
        &+\E\int_0^T\int_{G_0\cap\mathcal{M}}\left[\xi^2\varphi\partial_t(se^{2s\theta\varphi})|w|^2+2s\xi^2\varphi e^{2s\theta\varphi}\,wf+s\xi^2\varphi e^{2\theta\varphi}g^2\right]dt.
    \end{split}
\end{equation}
Using \eqref{eq:difference:product} and \eqref{eq:int:dif} on the last integral above, we obtain for each $i=1,...,n$ that
\begin{equation}\label{eq2:lemma4.7}
    \begin{split}
        &\E\int_0^T\int_{G_0\cap\mathcal{M}}s\xi^2 \varphi e^{2\theta\varphi}wD_i(\gamma_iD_iw)\,dt=-\E\int_0^T\int_{G_0\cap\mathcal{M}_{i}^{\ast}}s\gamma_iD_i(\xi^2 \varphi e^{2\theta\varphi}w)D_iw\,dt\\
        &+\E\int_{0}^T \int_{G_0\cap\partial_i\mathcal{M}}\xi^2 s\varphi e^{2\theta\varphi}w\,t_r^i(\gamma_iD_iw)\,dt=-\E\int_0^T\int_{G_0\cap\mathcal{M}_{i}^{\ast}}s\gamma_iA_i(\xi^2 \varphi e^{2\theta\varphi})\,|D_iw|^2\,dt\\
        &-\E\int_0^T\int_{G_0\cap\mathcal{M}_{i}^{\ast}}s\gamma_iD_i(\xi^2 \varphi e^{2\theta\varphi})A_iw\,D_iw\,dt,   
    \end{split}
\end{equation}
where used that $z=0$ on $\partial_i Q$. Therefore, combining \eqref{eq1:lemma4.7} and \eqref{eq2:lemma4.7} we conclude that
\begin{equation}\label{eq3:lemma4.7}
    \begin{split}
        2&\sum_{i=1}^{n}\E\int_0^T \int_{G_0\cap\mathcal{M}_i^{\ast}} s\gamma_i A_i(\xi^2\varphi e^{2\theta \varphi})\,|D_iw|^2\,dt+\E\int_0^T\int_{G_0\cap\mathcal{M}}s\xi^2\varphi e^{2\theta\varphi}g^2dt\\
        =& -\E\int_0^T\int_{G_0\cap\mathcal{M}}\left[\xi^2\varphi\partial_t(se^{2s\theta\varphi})|w|^2+2s\xi^2\varphi e^{2s\theta\varphi}\,wf\right]dt+\left.\E\int_{G_0\cap\mathcal{M}}s\varphi \xi^2 e^{2s\theta\varphi}|w|^2\right|_0^T\\
        &-2\sum_{i=1}^{n}\E\int_0^T\int_{G_0\cap\mathcal{M}_{i}^{\ast}}s\gamma_iD_i(\xi^2 \varphi e^{2\theta\varphi})A_iw\,D_iw\,dt.
    \end{split}
\end{equation}
By Proposition \ref{prop:weight}, for $\tau h(\delta T^2)^{-1}\leq 1$, it follows that $A_i(\xi^2\varphi e^{2\theta \varphi})=\xi^2\varphi e^{2\theta \varphi}+\mathcal{O}_{\lambda}((sh)^2)e^{2s\theta\varphi}$. Then the first term in the left-hand of \eqref{eq3:lemma4.7} for $i=1,...,n$ can be estimated by
\begin{equation}
    \begin{split}
         2\E\int_0^T \int_{G_0\cap\mathcal{M}_i^{\ast}} s\gamma_i A_i(\xi^2\varphi e^{2\theta \varphi})\,|D_iw|^2\,dt=& 2\E\int_0^T \int_{G_0\cap\mathcal{M}_i^{\ast}}s\xi^2\varphi e^{2\theta \varphi}\,|D_iw|^2\,dt\\
         &+\E\int_0^T \int_{G_0\cap\mathcal{M}_i^{\ast}}s\mathcal{O}_{\lambda}((sh)^2)e^{2s\theta\varphi}\,|D_iw|^2\,dt.
    \end{split}
\end{equation}
For the first term in the right-hand side of \eqref{eq3:lemma4.7}, we have
\begin{equation}
\E\int_0^T\int_{G_0\cap\mathcal{M}}\xi^2\varphi\partial_t(se^{2s\theta\varphi})|w|^2\,dt=\E\int_0^T\int_{G_0\cap\mathcal{M}}\theta\mathcal{O}_{\lambda}(s^2)|w|^2\,dt.
\end{equation}
Moreover, we see that
\begin{equation}\label{eq4:lemma4.7}
\begin{split}
    \left|\E\int_0^T\int_{G_0\cap\mathcal{M}}2s\xi^2\varphi e^{2s\theta\varphi}\,wf\,dt\right|\leq C \E\int_0^T&\int_{G_0\cap\mathcal{M}}s^2\lambda ^2\varphi^2e^{2s\theta \varphi}|w|^2\,dt\\
    &+C\E\int_0^T\int_{\mathcal{M}}\lambda^{-2}e^{2s\theta \varphi}\,|f|^2\,dt.
\end{split}
\end{equation}
Next, we estimate the last term on the right hand side of \eqref{eq3:lemma4.7}. By Proposition \ref{prop:weight}, we have $D_i(\xi^2\varphi e^{2s\theta\varphi})=\partial_i(\xi^2\varphi e^{2s\theta\varphi})+h\mathcal{O}_{\lambda}(sh)e^{2s\theta\varphi}=s\lambda\mathcal{O}(1)\xi \varphi^2e^{2s\theta\varphi}+h\mathcal{O}_{\lambda}(sh)e^{2s\theta\varphi}$, then for each $i=1,...,n$
\begin{equation}\label{eq5:lemma4.7}
    \begin{split}
        -2&\E\int_0^T\int_{G_0\cap\mathcal{M}_{i}^{\ast}}s\gamma_iD_i(\xi^2 \varphi e^{2\theta\varphi})A_iw\,D_iw\,dt\\
        =\,&\E\int_0^T\int_{G_0\cap\mathcal{M}_{i}^{\ast}}\left[s^2\lambda\mathcal{O}(1)\xi \varphi^2e^{2s\theta\varphi}+\mathcal{O}_{\lambda}((sh)^2)e^{2s\theta\varphi}\right]A_iw\,D_iw\,dt\\
        \leq\, & \E\int_0^T\int_{G_0\cap\mathcal{M}_{i}^{\ast}}\xi^2s\varphi e^{2s\theta\varphi}|D_iz|^2\,dt+C\E\int_0^T\int_{G_0\cap\mathcal{M}_{i}^{\ast}}s^3\lambda^2\varphi^3 e^{2s\theta\varphi}|A_iw|^2\,dt\\
        &+\E\int_0^T\int_{G_0\cap\mathcal{M}_{i}^{\ast}}\mathcal{O}_{\lambda}((sh)^2)e^{2s\theta\varphi}|D_iz|^2\,dt+\E\int_0^T\int_{G_0\cap\mathcal{M}_{i}^{\ast}}\mathcal{O}_{\lambda}((sh)^2)e^{2s\theta\varphi}|A_iw|^2\,dt
    \end{split}
\end{equation}
where Young's inequality was used. Similarly to \eqref{eq:estimates:I_{42(a)}}, and noting that $A_i(\varphi ^3e^{2s\theta\varphi})=\varphi ^3e^{2s\theta\varphi}+\mathcal{O}_{\lambda}((sh)^2)e^{2s\theta\varphi}$, $A_i(e^{2s\theta\varphi})=\mathcal{O}_{\lambda}(1)e^{2s\theta\varphi}$, we have
\begin{equation}\label{eq6:lemma4.7}
    \begin{split}
        \E\int_0^T&\int_{G_0\cap\mathcal{M}_{i}^{\ast}}s^3\lambda^2\varphi^3 e^{2s\theta\varphi}|A_iw|^2\,dt+\E\int_0^T\int_{G_0\cap\mathcal{M}_{i}^{\ast}}\mathcal{O}_{\lambda}((sh)^2)e^{2s\theta\varphi}|A_iw|^2\,dt\\
        \leq&  \E\int_0^T\int_{G_0\cap\mathcal{M}}s^3\lambda^2\varphi^3 e^{2s\theta\varphi}|w|^2\,dt+\E\int_0^T\int_{G_0\cap\mathcal{M}}\mathcal{O}_{\lambda}((sh)^2)e^{2s\theta\varphi}|w|^2\,dt.
    \end{split}
\end{equation}
Hence, combining  \eqref{eq5:lemma4.7} and \eqref{eq6:lemma4.7}, the last term in the left-hand of \eqref{eq3:lemma4.7} can be estimated by
\begin{equation}\label{eq7:lemma4.7}
    \begin{split}
        -2&\E\int_0^T\int_{G_0\cap\mathcal{M}_{i}^{\ast}}s\gamma_iD_i(\xi^2 \varphi e^{2\theta\varphi})A_iw\,D_iw\,dt\\
        \leq\, & \E\int_0^T\int_{G_0\cap\mathcal{M}_{i}^{\ast}}\xi^2s\varphi e^{2s\theta\varphi}|D_iz|^2\,dt+C\E\int_0^T\int_{G_0\cap\mathcal{M}}s^3\lambda^2\varphi^3 e^{2s\theta\varphi}|w|^2\,dt\\
        &+\E\int_0^T\int_{G_0\cap\mathcal{M}_{i}^{\ast}}\mathcal{O}_{\lambda}((sh)^2)e^{2s\theta\varphi}|D_iz|^2\,dt+\E\int_0^T\int_{G_0\cap\mathcal{M}}\mathcal{O}_{\lambda}((sh)^2)e^{2s\theta\varphi}|w|^2\,dt.
    \end{split}
\end{equation}
Therefore, combining \eqref{eq3:lemma4.7}-\eqref{eq4:lemma4.7},\eqref{eq7:lemma4.7}, we conclude the result of the Lemma \ref{lem:estimate:termDw}.
%%%%%%%%%%%%%%%%%%%%%%%%%%

%%%%%%%%%%%%%%%%%%%%%%%%%%%%%%%%%%%%%%%%%%%%%%%%%%%%%%%%%%%%%%%%%%%%%%%%%%%%%%%%%%%%%%%%%%%%%%%%%%%%%%%%%%%%%%%%%%%%%%%%%%%%%%%%%%%%%%%%%%%%%%%%%%%%%%%%%%%%%%%%%%%%%%%%%%%%%%%%%%%%%%%%%%%%%

\end{appendices}

\bibliographystyle{abbrv}
\bibliography{references}

\end{document}